\documentclass[11pt]{amsart}

\usepackage{amsmath,amssymb,amsthm}
\usepackage{hyperref}
\hypersetup{hidelinks}
\newcommand{\C}{\mathbb{C}}
\newcommand{\R}{\mathbb{R}}
\newcommand{\cO}{\mathcal{O}}
\newcommand{\cT}{\mathcal{T}}
\newcommand{\cM}{\mathcal{M}}
\newcommand{\cMs}{\mathcal{M}^{\mathrm{s}}}
\newcommand{\Fss}{\mathcal{F}^{\mathrm{ss}}}
\newcommand{\FS}{\mathcal{F}^{S}}
\newcommand{\Id}{\operatorname{Id}}
\newcommand{\red}{\mathrm{red}}
\newcommand{\pt}{\mathrm{pt}}

\DeclareMathOperator{\Aut}{Aut}
\DeclareMathOperator{\End}{End}
\DeclareMathOperator{\Spec}{Spec}
\DeclareMathOperator{\Hol}{Hol}
\DeclareMathOperator{\Lie}{Lie}
\DeclareMathOperator{\Diff}{Diff}
\DeclareMathOperator{\SL}{SL}
\DeclareMathOperator{\GL}{GL}
\DeclareMathOperator{\Spf}{Spf}
\DeclareMathOperator{\Pic}{Pic}
\DeclareMathOperator{\Tot}{Tot}
\DeclareMathOperator{\tr}{tr}
\DeclareMathOperator{\rank}{rank}
\DeclareMathOperator{\im}{im}

\theoremstyle{plain}
\newtheorem{thm}{Theorem}[section]
\newtheorem{lem}[thm]{Lemma}
\newtheorem{prop}[thm]{Proposition}
\newtheorem{cor}[thm]{Corollary}
\newtheorem{alphthm}{Theorem}

\theoremstyle{definition}
\newtheorem{defn}[thm]{Definition}

\title[Universal moduli of Higgs bundles]{An analytic construction of the universal moduli space of $\SL_r(\C)$-Higgs bundles}
\author{\textsf{Shiyu Cao}}
\address{Chern Institute of Mathematics, Nankai University}
\email{shiyucao@126.com}
\date{}

\DeclareMathOperator{\Stab}{Stab}
\newcommand{\cH}{\mathcal{H}}
\newcommand{\cP}{\mathcal{P}}
\newcommand{\fk}{\mathfrak{k}}
\newcommand{\fg}{\mathfrak{g}}
\newcommand{\W}[1]{W^{#1,2}}
\DeclareMathOperator{\Vol}{Vol}
\DeclareMathOperator{\id}{id}

\begin{document}
\begin{abstract}
We construct the moduli space of marked polystable $\SL_r(\C)$-Higgs bundles by analytic methods. The resulting space is normal and Hausdorff. Its local quadratic quotient models are products over Teichm\"uller space, so the projection is flat with normal fibers and is a holomorphic submersion on the stable locus. The stabilizer of each Kuranishi parameter is naturally isomorphic to the group of automorphisms of its Higgs bundle preserving the determinant trivialization. We prove agreement with the complex and unitary gauge quotient topologies and establish the coarse moduli property for semistable families over reduced analytic bases. We then prove a relative analytic--algebraic comparison on the entire polystable space and properness of the relative Hitchin map. The mapping class group and Higgs scaling act holomorphically. On the stable locus, the construction agrees with the joint moduli space of Collier--Toulisse--Wentworth.
\end{abstract}

\maketitle

\section{Introduction}\label{sec:introduction}

Higgs bundles on a compact Riemann surface admit both algebraic and gauge-theoretic moduli constructions. In degree zero, a Higgs bundle is polystable precisely when it admits a Hermitian metric solving the Hitchin equations. Nonabelian Hodge theory identifies their isomorphism classes with semisimple local systems~\cite{Hitchin87,Donaldson87,Corlette88,Simpson92}. The complex structure of the Higgs moduli space depends on the curve. Allowing the complex structure of the curve to vary leads to the universal moduli problem over Teichm\"uller space.

Relative algebraic Higgs moduli spaces are provided by Simpson's construction~\cite[Theorem~4.7]{SimpsonModuliPartI},~\cite[Corollary~6.7 and Proposition~9.7]{SimpsonModuliPartII}. Pulling back from a moduli space of curves with finite level structure, Alessandrini and Collier obtain a complex analytic Higgs moduli space over Teichm\"uller space whose fiber over a marked curve $X$ is the moduli space of polystable $\SL_r(\C)$-Higgs bundles on $X$, with a holomorphic mapping class group action~\cite[Theorem~7.5 and Section~7.4]{AlessandriniCollierMaximalComponents}. Thus a marked analytic Higgs moduli space over Teichm\"uller space is already known to exist.

For a fixed curve, Fan constructs the moduli space by Kuranishi slices and analytic Hilbert quotients, proves normality and quadratic local models, and obtains a biholomorphism with the algebraic moduli space~\cite[Theorems~A--C]{Fan2020Construction}.\footnote{Result numbers follow the preprint version whenever one is listed.} Ono describes the deformation complex of a curve together with a Higgs bundle~\cite{Ono2022Deformations}. When the Higgs bundle admits a harmonic metric, his product theorem identifies the Kuranishi germ of the pair with the product of the curve and Higgs Kuranishi germs~\cite[Theorems~1.1 and~5.3]{Ono2023Structure}.

In this paper, we carry out the relative analytic construction for Higgs bundles with a specified determinant trivialization, including the singularities represented by strictly polystable bundles. The term introduced by variation of the curve is exact for the harmonic-bundle differential $D'$. The harmonic identities then show that the Kuranishi obstruction is exactly quadratic and independent of the curve parameter. To compare the local analytic and gauge quotients, we use an equivariant Moser argument to relate the varying real moment maps to one fixed quadratic moment map. The resulting flow estimates are uniform in the curve parameter. Together with local completeness and control of stabilizers, these estimates allow us to glue the local quotients by holomorphic transitions. Local lifts near the graded object of a semistable bundle establish the coarse moduli property over reduced analytic bases, including families with a nonpolystable central fiber.

Theorem~\ref{thm:relative-algebraic-comparison} identifies this analytic space with the relative algebraic construction on the entire polystable space. We prove normality of the framed relative parameter schemes before passing to their quotients. This ensures that their universal families, pulled back to Teichm\"uller neighborhoods, have reduced bases and hence fall within the coarse moduli property proved here. Classifying maps in both directions then define the comparison biholomorphism. Proposition~\ref{prop:relative-hitchin-map} deduces properness of the relative Hitchin map.

Several recent works study the geometry of universal Higgs moduli spaces. Collier--Toulisse--Wentworth construct the joint moduli space of stable $G$-Higgs bundles as a complex orbifold and study its geometric structures~\cite[Theorem~4.6 and Section~4.8]{CTW25}; we identify their stable $\SL_r(\C)$ space with our stable locus. \'Alvarez-C\'onsul--Garcia-Fernandez--Garc\'ia-Prada--Trautwein define the universal quotient of polystable Higgs bundles and obtain K\"ahler fibrations and pseudo-K\"ahler structures on suitable smooth open subsets through coupled equations~\cite[Section~5.1 and Theorems~5.18, 5.22]{UniversalHitchinModuliSpaces}. Hitchin constructs the varying complex structure on the smooth universal moduli space using harmonic-map energy and a symplectic Ehresmann connection~\cite{HitchinUniversalHiggsModuliSpace}. The present paper establishes the analytic quotient models at all polystable $\SL_r(\C)$ classes, their topology, the coarse moduli property for reduced bases, and the relative analytic--algebraic comparison. In the fixed-curve setting, Fan's subsequent work establishes the orbit-type stratification and quasi-projectivity~\cite{Fan2020OrbitTypes,Fan2020QuasiProjectivity}.

\subsection{The analytic space and its local models}

We fix a closed, connected, oriented smooth surface $\Sigma$ of genus $g\geq2$, its Teichm\"uller space $\cT_g$, and an integer $r\geq1$. A marked $\SL_r(\C)$-Higgs bundle consists of a marked curve $X$, a rank-$r$ holomorphic bundle $E$, a specified holomorphic trivialization $\det E\simeq\cO_X$, and a trace-free Higgs field $\varphi\in H^0(X,\End E\otimes K_X)$. A marking is an orientation-preserving diffeomorphism $\mathfrak m:\Sigma\to X$, considered up to isotopy. Isomorphisms preserve the marking, determinant trivialization, and Higgs field.

We denote by $\cM$ the set of polystable isomorphism classes and by $p\colon\cM\to\cT_g$ the forgetful map. Stability is defined in Definition~\ref{chlocal:defn:stability}. All gauge groups and stabilizers include every connected component. For $r=1$, the specified determinant trivialization is an isomorphism $E\simeq\cO_X$ and $\varphi=0$, so $\cM=\cT_g$ and Higgs scaling is trivial.

Our first result equips this set with a complex analytic structure compatible with the quotients by the complex and unitary gauge groups.

\begin{alphthm}\label{thm:analytic-space}
The Kuranishi quotient charts endow $\cM$ with the structure of a reduced, Hausdorff, second-countable complex analytic space for which $p$ is holomorphic. The associated chart maps are open analytic embeddings, and the resulting structure is independent of the choice of slices.

The topology agrees with the quotient of polystable Higgs pairs by the complex gauge group and with the quotient of the solution space of the Hitchin equations by the unitary gauge group, both in the smooth category and for $\W{k}$ pairs with $\W{k+1}$ gauge transformations, $k\geq4$. The analytic structure is also independent of the Sobolev index.
\end{alphthm}

At a point $x=(X,E,\varphi)\in\cM$, let
\begin{equation*}
H_x=\{g\in\Aut(E,\varphi):\det g=1\},
\end{equation*}
and let $K_x$ be its unitary subgroup for a harmonic metric compatible with the determinant trivialization. Let $C_x^\bullet$ be the trace-free deformation complex of Definition~\ref{chlocal:defn:deformation-complex} and set $V_x=\mathbb H^1(C_x^\bullet)$. Using harmonic representatives and the degree-two harmonic projection $\mathsf H_{2,x}$, we set
\begin{equation*}
Q_x=\bigl\{v\in V_x:\mathsf H_{2,x}[v,v]=0\bigr\}_{\red},
\qquad Y_x=Q_x/\!/H_x.
\end{equation*}
Here $/\!/$ denotes the affine quotient in geometric invariant theory (GIT), and $Y_x$ is the analytification of $\Spec\C[Q_x]^{H_x}$. We choose a coordinate neighborhood $B\subset\cT_g$ of $X$ and represent $X$ by $0$.

\begin{alphthm}\label{thm:normal-model}
The space $\cM$ is normal. At each point $x$ there is an analytic germ isomorphism
\begin{equation*}
(\cM,x)\simeq_{(\cT_g,X)}
\bigl(B\times Y_x,(0,[0])\bigr).
\end{equation*}
For each parameter $(b,v)$ of a local Kuranishi family centered at $x$, the stabilizer $\Stab_{H_x}(v)$ is naturally isomorphic to the group of automorphisms of the Higgs bundle at $(b,v)$ preserving its determinant trivialization. The quotient $Y_x$ is normal at every point.
\end{alphthm}

The product over $\cT_g$ implies that $p$ is flat with normal fibers. On the stable locus, it is a holomorphic submersion whose vertical tangent space is $V_x$; the tangent sequence and dimensions are stated in Corollary~\ref{cor:flat-stable}.

\subsection{Families and natural actions}

\begin{defn}\label{defn:family-functors}
For a reduced finite-dimensional Hausdorff complex analytic space $S$, let $\Fss(S)$ be the set of isomorphism classes of holomorphic families of marked, smooth, proper curves of genus $g$ over $S$, equipped with families of semistable rank-$r$ Higgs bundles, determinant trivializations, and trace-free relative Higgs fields. Pullback defines a contravariant functor $\Fss$.

Two families are fiberwise S-equivalent if their polystable graded objects, with the induced markings and determinant trivializations, are isomorphic at every point of $S$. We write $\FS(S)$ for the quotient by this equivalence.
\end{defn}

The space satisfies the following coarse moduli property for semistable families.

\begin{alphthm}\label{thm:coarse}
There is a natural transformation $\eta\colon\Fss\to\Hol(-,\cM)$ assigning to a family its fiberwise polystable graded class. Its composition with $p$ is the classifying map of the underlying marked curve family. For $\mathcal E_1,\mathcal E_2\in\Fss(S)$, the equality $\eta_S(\mathcal E_1)=\eta_S(\mathcal E_2)$ holds exactly when the two families are fiberwise S-equivalent. Moreover, $\FS(\pt)\simeq\cM$.

For every Hausdorff complex analytic space $T$, each natural transformation $\Theta\colon\Fss\to\Hol(-,T)$ factors uniquely as $\Theta=f\circ\eta$ for a holomorphic map $f\colon\cM\to T$. The same property holds with $\FS$ as source.
\end{alphthm}

We also establish the natural actions and the comparison on the stable locus.

\begin{alphthm}\label{thm:actions-comparison}
The mapping class group $\Gamma_g=\pi_0(\Diff^+(\Sigma))$ acts biholomorphically on $\cM$, covering its action on $\cT_g$. Higgs scaling gives a holomorphic action
\begin{equation*}
\C^*\times\cM\longrightarrow\cM,
\qquad (\lambda,[X,E,\varphi])\longmapsto[X,E,\lambda\varphi].
\end{equation*}
These actions commute and preserve the automorphism groups under their natural identifications.

The stable locus $\cMs$ is open, with stabilizer $\{\zeta\Id_E:\zeta^r=1\}$. For $r\geq2$, it is biholomorphic over $\cT_g$ to the underlying analytic space of the stable joint $\SL_r(\C)$ moduli space constructed in~\cite{CTW25}. With $\Phi=2i\varphi$, the comparison is induced in both directions by holomorphic families and is equivariant for the two actions.
\end{alphthm}

\subsection*{Outline of the proof}
Section~\ref{sec:tech-chlocal} constructs the relative Kuranishi families and local quotients. Section~\ref{sec:tech-chtopology} proves agreement with the gauge quotient topologies using a moment-map flow whose estimates are uniform over compact subsets of the Teichm\"uller chart. Section~\ref{sec:transitions} glues the models and proves normality and the coarse moduli property. Section~\ref{sec:actions} treats the natural actions, the stable tangent sequence, and the comparison with the construction of Collier--Toulisse--Wentworth. Section~\ref{sec:algebraic-comparison} proves the relative algebraic comparison and properness of the relative Hitchin map.

\subsection*{Acknowledgments}
The author thanks Professor Qiongling Li for her support and Professors Kefeng Liu and Huitao Feng for their guidance over the years.

\subsection*{Use of AI tools}
The idea behind this work belongs to the author. During preparation and revision, ChatGPT was used for literature searches, editorial assistance, and checks of mathematical arguments and notation. The author takes full responsibility for the content.

\section{Relative deformation theory and local quotients}\label{sec:tech-chlocal}

We adapt the fixed-curve construction of~\cite{Fan2020Construction} to the marked Teichm\"uller family and record the arguments that depend on the curve parameter. In the local constructions we assume $r\geq2$; for $r=1$, the moduli space is $\cT_g$ as noted above.

\subsection{Relative equations and harmonic data}\label{chlocal:sec:preliminaries}\label{chlocal:sec:relative}\label{chlocal:sec:hodge}

\begin{defn}\label{chlocal:defn:stability}
For a nonzero torsion-free coherent sheaf $F$, we set $\operatorname{slope}(F)=\deg(F)/\rank(F)$. A Higgs bundle $(E,\varphi)$ is stable, respectively semistable, if every coherent subsheaf $F\subset E$ with $0<\rank F<\rank E$ and $\varphi(F)\subset F\otimes K_X$ satisfies $\operatorname{slope}(F)<\operatorname{slope}(E)$, respectively $\operatorname{slope}(F)\leq\operatorname{slope}(E)$. It is polystable if it is a direct sum of stable Higgs bundles of the same slope. Here $K_X$ is the canonical bundle and $\operatorname{slope}(E)=0$.
\end{defn}

For a Hermitian metric $h$, let $\nabla_h$ be the Chern connection and $\varphi^{*h}$ the adjoint $(0,1)$-form. By the Hitchin--Kobayashi correspondence, a degree-zero Higgs bundle is polystable precisely when it admits a Hermitian metric $h$ solving
\begin{equation}\label{chlocal:eq:hitchin}
\bar\partial_{\nabla_h}\varphi=0,\qquad F_{\nabla_h}+[\varphi,\varphi^{*h}]=0.
\end{equation}
The connection $\nabla_h+\varphi+\varphi^{*h}$ is flat with completely reducible monodromy~\cite{Hitchin87,Donaldson87,Corlette88,Simpson92}.

We fix $x=(X,E,\varphi)\in\cM$ and $k\geq4$. We pull back by the marking to work on $\Sigma$ and use $E$ for the underlying smooth bundle. We choose a K\"ahler form $\omega$ and a harmonic metric $h$ for which the chosen determinant trivialization is unitary. We write
\begin{equation*}
H_x=\{g\in\Aut(E,\varphi):\det g=1\}
\end{equation*}
and denote the complex and $h$-unitary $\W{k+1}$ gauge groups of determinant one by $\mathcal G_{k+1}^{\C}$ and $\mathcal G_{k+1}$, respectively. Unless stated otherwise, all norms, adjoints, and harmonic operators are defined using $\omega$, $h$, and a fixed smooth reference connection. We write $\|\cdot\|_{s,2}$ for the $\W{s}$ norm and $\|\cdot\|_2=\|\cdot\|_{0,2}$ for the $L^2$ norm.

Taking the trace of~\eqref{chlocal:eq:hitchin} shows that the determinant metric is constant; a constant rescaling gives the stated normalization and the determinant connection $d$.

\begin{defn}\label{chlocal:defn:deformation-complex}
We write $\End_0E=\ker(\tr\colon\End E\to\cO_X)$ and denote by $\bar\partial_{\End}$ and $\partial_h$ the operators induced on endomorphism-valued forms by the Chern connection. The trace-free deformation complex is $C_x^\bullet=[\End_0E\xrightarrow{[\varphi,\cdot]}\End_0E\otimes K_X]$, in degrees zero and one. Its Dolbeault resolution is
\begin{equation*}
A^0(\End_0E)\xrightarrow{D_0}A^{0,1}(\End_0E)\oplus A^{1,0}(\End_0E)\xrightarrow{D_1}A^{1,1}(\End_0E),
\end{equation*}
where $D_0u=(\bar\partial_{\End}u,[\varphi,u])$ and $D_1(a,\psi)=\bar\partial_{\End}\psi+[\varphi,a]$. For forms of total degrees $i,j$, we set $[S,T]=S\wedge T-(-1)^{ij}T\wedge S$ and $V_x=\mathbb H^1(C_x^\bullet)$.
\end{defn}

Its cohomology describes infinitesimal automorphisms, deformations, and obstructions~\cite[Theorems~3.1, 3.6]{Ono2022Deformations}.

Deformations of the marked curve are represented by harmonic Beltrami differentials. We denote the space of harmonic representatives of $H^{0,1}(X,T_X)$ by $\mathcal B\subset A^{0,1}(X,T_X)$ and put $\mu(b)=b$ for $b\in\mathcal B$.

\begin{lem}\label{chlocal:lem:marked-base}
Let $B\subset\mathcal B$ be a sufficiently small open ball centered at the origin, with $\|\mu(b)\|_\infty<1$ for every $b\in B$. Then $B$ is a Teichm\"uller chart and carries the universal marked curve family over analytic bases. In a holomorphic coordinate $z$ on the central fiber, the vertical $(0,1)$-direction is spanned by $\partial_{\bar z}-\mu(b)\partial_z$; the base directions are $\partial_{\bar b_j}$. After identifying the marked curve families, isomorphisms of the Higgs families are represented by bundle isomorphisms covering the identity on the curve family.
\end{lem}

The Kodaira--Spencer map at the origin is the isomorphism $T_0B=\mathcal B\to H^{0,1}(X,T_X)$ given by harmonic representatives. By the holomorphic inverse function theorem, its classifying map is a biholomorphism onto a Teichm\"uller neighborhood after shrinking $B$. The local universal property for marked curve families is established in~\cite[Theorem~3.1, Remark~3.2]{Grothendieck1961Teich}. For orientation-preserving homeomorphisms of $\Sigma$, homotopy and isotopy define the same equivalence relation on markings~\cite[Theorem~1.12]{FarbMargalitPrimer}.

We use relative coordinates $(b,\alpha)$ with
\begin{equation*}
\alpha=(a,\psi),\qquad
(a,\psi)\in \W{k}\bigl(A^{0,1}(\End_0E)\oplus A^{1,0}(\End_0E)\bigr).
\end{equation*}
We put $\varphi_{\mathrm{tot}}=\varphi+\psi$.

\begin{prop}\label{chlocal:prop:relative-equations}
The relative Higgs field is $\varphi_{\mathrm{tot}}+\iota_{\mu(b)}\varphi_{\mathrm{tot}}$, and its holomorphicity equation is
\begin{equation}\label{chlocal:eq:relative-equation}
F(b,\alpha)=\bar\partial_{\End}\varphi_{\mathrm{tot}}+\partial_h(\iota_{\mu(b)}\varphi_{\mathrm{tot}})+[a,\varphi_{\mathrm{tot}}]=0.
\end{equation}
The complex gauge group acts holomorphically on the right by
\begin{equation}\label{chlocal:eq:right-gauge}
a^g=g^{-1}ag+g^{-1}(\bar\partial_{\End}g-\iota_{\mu(b)}\partial_hg),\qquad
\varphi_{\mathrm{tot}}^g=g^{-1}\varphi_{\mathrm{tot}}g.
\end{equation}
This action satisfies
\begin{equation*}
F(b,\alpha^g)=g^{-1}F(b,\alpha)g.
\end{equation*}
Since it is a right action, for $g,\sigma\in\mathcal G_{k+1}^{\C}$ one has
\begin{equation}\label{chlocal:eq:right-action-composition}
(\alpha^g)^\sigma=\alpha^{g\sigma}.
\end{equation}
\end{prop}

\begin{proof}
We write $\varphi_{\mathrm{tot}}=\varphi_{\mathrm{tot},z}\,dz$. The vertical Dolbeault operator is $L=\bar\partial_E-\iota_{\mu(b)}\partial_h+a$, and the relative Higgs field is $\varphi_{\mathrm{tot},z}(dz+\mu(b)\,d\bar z)$. In the chosen frame, the Chern connection has the form $\nabla_h=d+\Gamma_z\,dz+\Gamma_{\bar z}\,d\bar z$. The holomorphicity equation has coefficient
\begin{equation*}
\partial_{\bar z}\varphi_{\mathrm{tot},z}
-\partial_z\bigl(\mu(b)\varphi_{\mathrm{tot},z}\bigr)
+[\Gamma_{\bar z}-\mu(b)\Gamma_z+a_{\bar z},\varphi_{\mathrm{tot},z}]
\end{equation*}
in the order $d\bar z\wedge dz$, giving~\eqref{chlocal:eq:relative-equation}. Vertical integrability holds because $A^{0,2}(X_b)=0$. Conjugation of $L$ and $\varphi_{\mathrm{tot}}$ yields the action formula and its covariance. By Sobolev multiplication, $F\colon B\times \W{k}\to \W{k-1}$ is a continuous polynomial, and the multiplication and inversion maps defining the gauge action are holomorphic~\cite[Chapter~13, Sections~1--3]{Taylor2011PDEIII}.
\end{proof}

The linearized equation and gauge action are
\begin{equation}\label{chlocal:eq:reduced-differential}
u\longmapsto(D_0u,0),\qquad
(\alpha,\dot b)\longmapsto D_1\alpha+M(\dot b),\qquad
M(\dot b)=\partial_h(\iota_{\dot b}\varphi).
\end{equation}
Fixing the marking reduces Ono's joint deformation complex to~\eqref{chlocal:eq:reduced-differential}~\cite[Theorems~3.1, 3.6]{Ono2022Deformations}. Hodge decomposition and $H^0(X,T_X)=0$ make the quotient
\begin{equation*}
A^0(T_X)\xrightarrow{\bar\partial}A^{0,1}(T_X)/\mathcal B
\end{equation*}
acyclic. Moreover, $A^{0,2}(T_X)=0$ on a Riemann surface, so the Kodaira--Spencer equation is unobstructed. Thus $B$ is a local universal deformation space for the marked curve, while $\mathbb H^2(C_x^\bullet)$ contains the obstructions to deforming the Higgs bundle.

We call a Dolbeault operator together with a Higgs field a Higgs pair. For a fixed metric $h$, it is equivalently represented by its Chern connection and Higgs field. The topological comparison uses this representation over $B$.

\begin{defn}\label{chlocal:defn:unitary-pairs}
For a smooth family $J_b$ of complex structures, let $\operatorname{Conn}_k(E,h)$ be the affine space of $\W{k}$ unitary connections inducing $d$ on $\det E$. We set
\begin{equation*}
\mathcal C_{B,k}=\{(b,\nabla,\varphi):b\in B,\ \nabla\in\operatorname{Conn}_k(E,h),\ \varphi\in \W{k}(T^{*1,0}_{J_b}\Sigma\otimes\End_0E)\},
\end{equation*}
and endow it with the subspace topology induced from
\begin{equation*}
B\times\operatorname{Conn}_k(E,h)\times
\W{k}(T^*\Sigma\otimes_{\R}\C\otimes\End_0E).
\end{equation*}
\end{defn}

\begin{prop}\label{chlocal:prop:unitary-properness}
The unitary gauge action on $\mathcal C_{B,k}$ is continuous and proper. Every invariant subset, with its induced topology, has an open quotient map and Hausdorff quotient. Smooth reference metrics give uniformly equivalent norms over relatively compact subsets of $B$.
\end{prop}

\begin{proof}
The gauge action on the connection factor is proper~\cite[Section~3.1]{Fan2020Construction}. The gauge transformations fix $b$, and the subgroup of elements of determinant one is closed. Restrictions to invariant subspaces remain proper; the quotient assertions follow from properness and openness of group actions. Norm equivalence follows from compactness.
\end{proof}

The metric $h$ assigns a unique unitary connection to each Dolbeault operator. Together with the map $\varphi_{\mathrm{tot}}\mapsto\varphi_{\mathrm{tot}}+\iota_{\mu(b)}\varphi_{\mathrm{tot}}$, these assignments give a bijection between the coordinates $(b,\alpha)$ and $\mathcal C_{B,k}$. We transport the complex gauge action~\eqref{chlocal:eq:right-gauge} through this bijection: for $g\in\mathcal G_{k+1}^{\C}$, the pair $(\nabla,\varphi)^g=(\nabla^g,g^{-1}\varphi g)$ is determined by
\begin{equation*}
(\nabla^g)^{0,1}_{J_b}=g^{-1}\circ\bar\partial_\nabla\circ g,
\end{equation*}
where $\nabla^g$ is the unitary connection associated to this Dolbeault operator and $h$. On the unitary subgroup, this is the usual gauge action. This correspondence and its inverse depend smoothly on $b$ and are continuous at every Sobolev order, with operator bounds uniform on relatively compact subsets of $B$.

We write $D''=\bar\partial_{\End}+[\varphi,\cdot]$, $D'=\partial_h+[\varphi^{*h},\cdot]$, and $D=D'+D''$. On the trace-free complex, we set
\begin{equation*}
\Delta_i=D_{i-1}D_{i-1}^*+D_i^*D_i,\qquad H^i=\ker\Delta_i,
\end{equation*}
with absent arrows omitted. The map taking a harmonic form to its cohomology class is an isomorphism $H^i\simeq\mathbb H^i(C_x^\bullet)$. We use it throughout the local construction to write $V_x=H^1$. We denote the harmonic projections by $\mathsf H_i$ and the Green operators by $G_i$; they satisfy $\Delta_iG_i=1-\mathsf H_i$ and $G_iH^i=0$.

Let $\Delta'$, $\Delta''$, and $\Delta_D$ denote the Laplacians of $D'$, $D''$, and $D$, respectively; thus $\Delta_i$ is the degree-$i$ component of $\Delta''$. The harmonic-bundle identities are
\begin{equation}\label{chlocal:eq:kahler-identities}
(D')^*=i[\Lambda,D''],\qquad (D'')^*=-i[\Lambda,D'],\qquad \Delta'=\Delta''=\tfrac12\Delta_D,
\end{equation}
where $\Lambda$ denotes contraction with $\omega$. We also use the usual Hodge decomposition and the estimates
\begin{equation*}
G_i\colon \W{s}\longrightarrow \W{s+2};
\end{equation*}
see~\cite[Section~4.2, Lemma~4.1]{Ono2023Structure}. All these operators preserve the trace-free summand $\End_0E$.

\begin{lem}\label{chlocal:lem:harmonic-equivariance}
Every holomorphic Higgs endomorphism is parallel with respect to $\nabla_h$ and commutes with $\varphi^{*h}$. Its adjoint is also a holomorphic Higgs endomorphism. Conjugation by $H_x$ commutes with the differentials, adjoints, harmonic projections, and Green operators, and fixes $M(\mathcal B)$.
\end{lem}

\begin{proof}
If $D''f=0$, the equality of degree-zero Laplacians implies $D'f=0$. Separating types proves the first assertions, which also apply to $f^{*h}$. Conjugation by $g\in H_x$ commutes with $D'$ and $D''$; since $g^{*h}\in H_x$, it also commutes with their Hilbert adjoints, harmonic projections, and Green operators. Parallelism and $[g,\varphi]=0$ imply
\begin{equation*}
g^{-1}M(\dot b)g
=g^{-1}\partial_h(\iota_{\dot b}\varphi)g
=\partial_h(\iota_{\dot b}\varphi)=M(\dot b).
\end{equation*}
\end{proof}

\begin{lem}\label{chlocal:lem:hodge-inverse}
On $(\im D_0\oplus\im D_1^*)\cap \W{k}$, the bounded inverse of $(D_0^*,D_1)$ is
\begin{equation*}
(f,z)\longmapsto D_0G_0f+D_1^*G_2z,
\end{equation*}
for $f,z\in \W{k-1}$ orthogonal to $H^0,H^2$, respectively. Moreover $\mathsf H_2M=0$. The cohomology of~\eqref{chlocal:eq:reduced-differential} is $H^0$, $V_x\oplus\mathcal B$, $H^2$, with equivariant degree-one representatives
\begin{equation*}
(v-D_1^*G_2M(\dot b),\dot b),\qquad v\in V_x,\quad\dot b\in\mathcal B.
\end{equation*}
\end{lem}

\begin{proof}
The inverse and closed ranges follow from Hodge decomposition. Indeed,
\begin{equation*}
D_1D_0=0,\qquad D_0^*D_1^*=0,\qquad
\Delta_0=D_0^*D_0,\qquad \Delta_2=D_1D_1^*,
\end{equation*}
so $(D_0^*,D_1)(D_0G_0f+D_1^*G_2z)=(f,z)$. On a curve,
\begin{equation*}
M(\dot b)=D'(\iota_{\dot b}\varphi),
\end{equation*}
so $\mathsf H_2M=0$. Thus $D_1D_1^*G_2M=M$ and $D_0^*D_1^*G_2M=0$, giving the asserted cohomology and representatives.
\end{proof}

We call $D_0^*\alpha=0$ the Coulomb gauge condition. For $\alpha=(a,\psi)$, the linear operator governing regularity is
\begin{equation*}
\mathcal D_b\alpha=(D_0^*\alpha,D_1\alpha+\partial_h(\iota_{\mu(b)}\psi)).
\end{equation*}

\begin{lem}\label{chlocal:lem:uniform-estimates}
For every integer $s\geq0$ and every relatively compact subchart $B_0\Subset B$, there is a constant $C_{s,B_0}>0$ such that
\begin{equation*}
\|\alpha\|_{s+1,2}\leq C_{s,B_0}(\|\mathcal D_b\alpha\|_{s,2}+\|\alpha\|_{0,2})
\end{equation*}
for all $b\in B_0$ and $\alpha\in \W{s+1}$. The analogous estimate holds for $u\mapsto(\bar\partial_{\End}u-\iota_{\mu(b)}\partial_hu,[\varphi,u])$. A representative in Coulomb gauge with $F(b,\alpha)\in H^2$ is smooth.
\end{lem}

\begin{proof}
The symbols are elliptic for $\|\mu(b)\|_\infty<1$, uniformly on compact subsets of $B$. The bound follows from the parameter-dependent elliptic estimate. Since
\begin{equation*}
\mathcal D_b\alpha=(0,F(b,\alpha)-M(\mu(b))-[a,\psi]),
\end{equation*}
Sobolev multiplication and elliptic iteration imply smoothness.
\end{proof}

\subsection{Relative Kuranishi families}\label{chlocal:sec:families}

We now construct the relative Kuranishi family and its obstruction map by the holomorphic implicit function theorem.

\begin{prop}\label{chlocal:prop:kuranishi-reduction}
On a product $\Omega=B\times U$ of balls centered at the origin, $U\subset V_x$, there are holomorphic maps
\begin{equation*}
w\colon\Omega\longrightarrow\im D_1^*\cap \W{k},
\qquad
\kappa\colon\Omega\longrightarrow H^2.
\end{equation*}
Setting $A(b,v)=v+w(b,v)$, one has
\begin{equation*}
D_0^*A=0,\qquad\mathsf H_1A=v,\qquad F(b,A)=\kappa(b,v),
\end{equation*}
and
\begin{equation}\label{chlocal:eq:kuranishi-derivative}
A(0,0)=0,\quad dA_0(\dot b,\dot v)=\dot v-D_1^*G_2M(\dot b),\quad \kappa(0)=d\kappa_0=0.
\end{equation}
There are an open neighborhood $\mathcal V$ of $(0,0)$ in the ambient space of pairs $(b,\alpha)$ and $\varepsilon>0$ such that the equation
\begin{equation*}
D_0^*\bigl(\alpha^{\exp u(b,\alpha)}\bigr)=0
\end{equation*}
has a unique solution $u(b,\alpha)\in(H^0)^\perp\cap \W{k+1}$ with $\|u(b,\alpha)\|_{k+1,2}<\varepsilon$. This solution depends holomorphically on $(b,\alpha)$ and satisfies $u(0,0)=0$. The Kuranishi parameter map
\begin{equation}\label{chlocal:eq:classifying-map}
c_x(b,\alpha)=\left(b,\mathsf H_1\bigl(\alpha^{\exp u(b,\alpha)}\bigr)\right)
\end{equation}
is holomorphic on $\mathcal V$ and takes values in $\Omega$. For $(b,\alpha)\in\mathcal V\cap F^{-1}(0)$, it satisfies
\begin{equation*}
\alpha^{\exp u(b,\alpha)}=A\bigl(c_x(b,\alpha)\bigr),
\qquad \kappa\bigl(c_x(b,\alpha)\bigr)=0.
\end{equation*}
\end{prop}

\begin{proof}
The holomorphic implicit function theorem, applied as in~\cite[Section~4]{Ono2022Deformations} to
\begin{equation*}
(1-\mathsf H_2)F(b,v+w)=0,
\qquad w\in D_1^*\W{k+1},
\end{equation*}
gives the map $w$. The derivative in the $w$-direction is $D_1$, with inverse $D_1^*G_2$. We set $\kappa=\mathsf H_2F(b,A)$; differentiation and $\mathsf H_2M=0$ give~\eqref{chlocal:eq:kuranishi-derivative}.

Gauge fixing is obtained by solving
\begin{equation*}
D_0^*(\alpha^{\exp u})=0,
\qquad u\in(1-\mathsf H_0)\W{k+1}.
\end{equation*}
At $(b,\alpha,u)=(0,0,0)$, the derivative in the $u$-direction is $\Delta_0$, with inverse $G_0$. The implicit function theorem gives $u$ on an open neighborhood in the ambient space of pairs. After shrinking this neighborhood, the Kuranishi parameter map~\eqref{chlocal:eq:classifying-map} takes values in $\Omega$. For $F(b,\alpha)=0$, covariance and uniqueness in the first construction give the two asserted identities. We may shrink $\Omega$ and $\mathcal V$ compatibly so that $(b,A(b,v))\in\mathcal V$ for every $(b,v)\in\Omega$. The Coulomb gauge condition and uniqueness then yield $u(b,A(b,v))=0$.
\end{proof}

\begin{lem}
\label{chlocal:lem:kuranishi-regularity}
For every integer $m\geq0$, the map
\begin{equation*}
A\colon\Omega\longrightarrow
\W{m}\bigl(A^{0,1}(\End_0E)\oplus A^{1,0}(\End_0E)\bigr)
\end{equation*}
is holomorphic, and its values are smooth on $\Sigma$. The germ of $A$ is independent of the Sobolev index $k\geq4$. The gauge-fixing maps are compatible with Sobolev inclusions: for $k'\geq k$, the maps constructed at these two indices agree on a common neighborhood of the origin in $B\times \W{k'}$. For $\sigma\in H_x$ and $v^\sigma=\sigma^{-1}v\sigma$,
\begin{equation}\label{chlocal:eq:restricted-equivariance}
A(b,v^\sigma)=A(b,v)^\sigma,\qquad
\kappa(b,v^\sigma)=\kappa(b,v)^\sigma
\end{equation}
whenever $(b,v)$ and $(b,v^\sigma)$ lie in $\Omega$.
The map
\begin{equation*}
(\sigma,b,v)\longmapsto(b,v^\sigma)
\end{equation*}
is holomorphic on the open set
\begin{equation*}
\{(\sigma,b,v)\in H_x\times\Omega:(b,v^\sigma)\in\Omega\}.
\end{equation*}
\end{lem}

\begin{proof}
Writing $A=(a,\psi)$, Lemma~\ref{chlocal:lem:uniform-estimates} and Sobolev multiplication applied to
\begin{equation*}
\mathcal D_bA=(0,\kappa-M(\mu(b))-[a,\psi])
\end{equation*}
give, recursively,
\begin{equation*}
\sup_{(b,v)\in K}\|A(b,v)\|_{m,2}<\infty,
\qquad K\Subset\Omega,\quad m\geq k.
\end{equation*}
Interpolation proves continuity into $\W{m}$, so the Banach-valued Cauchy formula proves holomorphicity in this space. The lower Sobolev orders follow by inclusion, while Cauchy estimates on compact subsets of $\Omega$, together with Sobolev embedding, imply joint smoothness.

A gauge transformation $g$ with $(b,\alpha)^g=(b,\alpha')$, where $\alpha=(a,\psi)$ and $\alpha'=(a',\psi')$ are smooth, satisfies
\begin{equation*}
(\bar\partial_{\End}-\iota_{\mu(b)}\partial_h)g=ga'-ag.
\end{equation*}
By elliptic regularity in degree zero, $g$ belongs to every Sobolev space and is therefore smooth. For $k'\geq k\geq4$, the two constructions by the implicit function theorem use the same orthogonal complements of the harmonic spaces. We restrict their finite-dimensional parameter domains to a common product neighborhood; uniqueness in $\W{k}$ shows that the two maps $A$ agree there. For the gauge-fixing maps, the inclusion $\W{k'}\hookrightarrow \W{k}$ is continuous. We may therefore restrict to a neighborhood in $B\times \W{k'}$ on which both maps are defined and the higher-order solution lies in the uniqueness neighborhood for the lower-order construction. Uniqueness in $\W{k+1}$ proves their equality there.

By Lemma~\ref{chlocal:lem:harmonic-equivariance}, every $\sigma\in H_x$ is parallel and commutes with $\varphi$. The action in~\eqref{chlocal:eq:right-gauge} therefore reduces to conjugation and preserves the equation and the chosen orthogonal complements. For fixed $\sigma$, the uniqueness part of the implicit function theorem proves~\eqref{chlocal:eq:restricted-equivariance} on a neighborhood of $(0,0)$. The domain
\begin{equation*}
\Omega_\sigma=\{(b,v)\in\Omega:(b,v^\sigma)\in\Omega\}
\end{equation*}
is an intersection of convex product domains containing the origin. Both sides are holomorphic there, so the identity theorem extends the equality to $\Omega_\sigma$.

The units of the finite-dimensional algebra $\C\Id_E\oplus H^0$ form a complex Lie group. Its subgroup of elements of determinant one is a complex Lie group since the determinant has nonzero derivative in the scalar direction. The conjugation map is holomorphic on the domain in the statement.
\end{proof}

\begin{lem}
\label{chlocal:lem:fiberwise-holomorphicity}
Let $T\subset\Delta$ be a reduced analytic subspace of a polydisc, and let the marked curve family over $\Delta$ be a pullback of the marked curve family in Lemma~\ref{chlocal:lem:marked-base}, with its induced smooth identification with $\Sigma\times\Delta$. Let $f$ be a locally defined function on its total space. Suppose that, on each relatively compact coordinate disc $W\subset\Sigma$ in its domain, the map $t\mapsto f(\cdot,t)|_{\overline W}$ is holomorphic with values in $C^1(\overline W)$. If $f|_T$ is holomorphic on every fiber, then it is holomorphic on the total space of the curve family over $T$. The same statement applies entrywise to matrices representing bundle homomorphisms in holomorphic frames and satisfying these hypotheses.
\end{lem}

\begin{proof}
A relative holomorphic coordinate $\zeta=\zeta(p,t)$ is holomorphic in $t$ for fixed $p$, because the total $(0,1)$-distribution contains $\partial_{\bar t_j}$. We choose a smooth simple closed curve $\Gamma\subset\Sigma$ bounding a coordinate disc in the central fiber. After shrinking $\Delta$, the curves $\zeta(\Gamma,t)$ enclose a fixed disc in the $\zeta$-plane for every $t\in\Delta$. For $f$ as in the statement, the integral
\begin{equation*}
\widetilde f(\zeta,t)=\frac{1}{2\pi i}\int_\Gamma
\frac{f(p,t)\,d_p \zeta(p,t)}{\zeta(p,t)-\zeta}
\end{equation*}
is holomorphic in $(\zeta,t)$. Indeed, the denominator stays nonzero on the compact contour, and the $C^1$-valued holomorphic dependence bounds the integrand and its parameter derivatives locally uniformly, justifying differentiation under the integral. For $t\in T$, the one-variable Cauchy formula identifies $\widetilde f(\,\cdot\,,t)$ with $f(\,\cdot\,,t)$. The pullback curve family is smooth over the reduced base $T$ and hence has reduced total space. Local extensions agreeing pointwise there agree as holomorphic functions, so they glue. The argument applies entrywise to matrices representing bundle homomorphisms.
\end{proof}

To express a holomorphic family in the relative coordinates, we construct a smooth local trivialization over $E$.

\begin{lem}
\label{chlocal:lem:smooth-trivialization}
Let $(S,s_0)$ be a reduced analytic germ carrying a marked curve family whose classifying map takes values in $B$, together with a family of trace-free Higgs bundles with a determinant trivialization. Any specified smooth identification of its bundle on the central fiber with $E$ preserving the determinant trivialization extends, after shrinking, to a smooth trivialization preserving the determinant trivialization, whose coefficients $(a(s),\psi(s))$ are holomorphic as $\W{m}$-valued maps for every finite $m$. At $s_0$, they agree with the prescribed coefficients. Holomorphicity on $S$ is understood through a local embedding in a polydisc and holomorphic extensions to that polydisc.
\end{lem}

\begin{proof}
Let $b\colon S\to B$ be the classifying map supplied by the local universal property of the Teichm\"uller family. We choose an embedding of $S$ into a polydisc and extend $b$ holomorphically, so that the marked curve family is defined over the polydisc.

After shrinking the polydisc, we choose finitely many relative holomorphic coordinate charts in its curve family, together with smaller patches covering the central fiber, on which the bundle over $S$ has holomorphic frames. The ambient charts and their intersections are Stein, and their pullbacks to $S$ are closed analytic subspaces. Cartan's theorem B therefore extends the entries of the transition matrices, the local functions defining the determinant trivialization, and the Higgs-field coefficients to these ambient charts or intersections. In the fixed smooth identification with $\Sigma$, these extensions are holomorphic in the polydisc parameter because the total $(0,1)$-distribution contains the parameter directions $\partial_{\bar s_j}$. On the smaller patches, compactness gives uniform bounds for every fixed number of derivatives along $\Sigma$. The cocycle identities are used only after restriction to $S$.

Using the prescribed smooth isomorphism, we regard the bundle on the central fiber as $E$. We fix a smooth partition of unity $(\varrho_i)$ subordinate to the patches and denote the frame coordinate maps by $\phi_i(s)$. If there are $N$ patches, we define
\begin{equation*}
\mathcal I_s:(E_s)_y\longrightarrow\C^{rN},
\qquad e\longmapsto(\varrho_i(y)\phi_i(s)e)_i.
\end{equation*}
We choose a smooth projection $\operatorname{pr}_0$ onto $\im\mathcal I_{s_0}$ and set
\begin{equation*}
\Xi_s:E_s\longrightarrow E,
\qquad \Xi_s=\mathcal I_{s_0}^{-1}\operatorname{pr}_0\mathcal I_s.
\end{equation*}
Each term in $\mathcal I_s$ is extended by zero outside its patch. The map $\mathcal I_s$ is injective and $\Xi_{s_0}=\Id_E$. By continuity in operator norm and compactness of $\Sigma$, $\Xi_s$ remains invertible after shrinking the parameter neighborhood. Its local matrices are smooth in $y$ and holomorphic in $s$.

Relative to the chosen determinant trivializations, the function $\det\Xi_s$ equals $1$ at $s=s_0$. After a further shrinking of $S$, its image lies in a simply connected disc centered at $1$. We choose the branch of $(\det\Xi_s)^{-1/r}$ that equals $1$ at $s_0$ and set
\begin{equation*}
\widehat\Xi_s=(\det\Xi_s)^{-1/r}\Xi_s,
\qquad
\det\widehat\Xi_s=1.
\end{equation*}
The chosen root is smooth in $y$ and holomorphic in $s$.

Using $\widehat\Xi_s$, we obtain
\begin{equation*}
\begin{aligned}
\widehat\Xi_s\bar\partial_{E_s}\widehat\Xi_s^{-1}
&=\bar\partial_E-\iota_{\mu(b(s))}\partial_h+a(s),\\
\widehat\Xi_s\varphi_s\widehat\Xi_s^{-1}
&=\varphi_{\mathrm{tot}}(s)+\iota_{\mu(b(s))}\varphi_{\mathrm{tot}}(s),
\end{aligned}
\end{equation*}
where $\varphi_{\mathrm{tot}}(s)=\varphi+\psi(s)$.
In the chosen frames, these coefficients are finite expressions in the extended matrices, their derivatives along $\Sigma$, and their inverses. The patchwise expressions agree over $S$. A smooth partition of unity in the fixed bundle, followed by the trace-free projection, extends these coefficients globally over the parameter polydisc. Uniform bounds on the closures of the coordinate patches allow us to apply the Banach-valued Cauchy formula, proving holomorphicity into every $\W{m}$. The construction preserves the prescribed values at $s_0$.
\end{proof}

We set $Z_x=(\kappa^{-1}(0))_{\red}$ and write
\begin{equation*}
\operatorname{pr}_B(b,v)=b,
\qquad
\theta_x(b,v)=(b,A(b,v)).
\end{equation*}

\begin{thm}\label{chlocal:thm:relative-family}
The map $\theta_x$ defines a holomorphic marked Higgs family on $Z_x$. Every family over a reduced germ $(S,s_0)$, with curve classifying map $b\colon S\to B$ and a chosen isomorphism of its central fiber with $x$, is locally isomorphic to $\gamma^*\theta_x$ for a holomorphic map $\gamma\colon S\to Z_x$ satisfying $\operatorname{pr}_B\circ\gamma=b$. The family isomorphism preserves the determinant trivialization and the chosen isomorphism on the central fiber. For each $\sigma\in H_x$, the action on the common domain of $(b,v)$ and $(b,v^\sigma)$ is induced by a holomorphic isomorphism between the two pullback families preserving the determinant trivialization.
\end{thm}

\begin{proof}
Over $\Omega$, we write
\begin{equation*}
A(b,v)=\bigl(a(b,v),\psi(b,v)\bigr),
\qquad
\varphi_{\mathrm{tot}}(b,v)=\varphi+\psi(b,v),
\end{equation*}
and set
\begin{equation*}
L_{b,v}=\bar\partial_E-\iota_{\mu(b)}\partial_h+a(b,v),
\qquad
\varphi^{\mathrm{rel}}_{b,v}=\varphi_{\mathrm{tot}}(b,v)+\iota_{\mu(b)}\varphi_{\mathrm{tot}}(b,v).
\end{equation*}
By Lemma~\ref{chlocal:lem:kuranishi-regularity}, the coefficients are smooth along the fibers and holomorphic in $(b,v)$ with values in every $\W{m}$. Combining $L_{b,v}$ with the standard $\bar\partial$-operator in the $B\times U$ directions defines a total $(0,1)$-operator. Holomorphic dependence in the parameter directions and $\dim_{\C}X=1$ imply that its square vanishes. It therefore defines a holomorphic bundle on the pullback of the marked curve family to $\Omega$~\cite[proof of Corollary~4.1]{Ono2022Deformations}; the chosen determinant trivialization is holomorphic because $\tr a(b,v)=0$. On $Z_x$, the relative Dolbeault derivative of $\varphi^{\mathrm{rel}}_{b,v}$ vanishes because
\begin{equation*}
F(b,A(b,v))=0.
\end{equation*}
Lemma~\ref{chlocal:lem:fiberwise-holomorphicity} therefore shows that $\varphi^{\mathrm{rel}}_{b,v}$ is a holomorphic relative Higgs field. These data define $\theta_x$. For $\sigma\in H_x$, equation~\eqref{chlocal:eq:restricted-equivariance} is induced by a holomorphic isomorphism between the corresponding pullback families on their common domain.

For a family as in the theorem, Lemma~\ref{chlocal:lem:smooth-trivialization} gives coefficients $(b(s),\alpha(s))$ with $\alpha(s_0)=0$. Proposition~\ref{chlocal:prop:kuranishi-reduction} gives the holomorphic map
\begin{equation*}
\gamma(s)=c_x(b(s),\alpha(s))
=\left(b(s),\mathsf H_1\bigl(\alpha(s)^{\exp u(b(s),\alpha(s))}\bigr)\right).
\end{equation*}
Gauge covariance implies $\kappa\circ\gamma=0$, so $\gamma$ factors through $Z_x$ because $S$ is reduced.

With the trivialization $\widehat\Xi_s$ from Lemma~\ref{chlocal:lem:smooth-trivialization}, the family isomorphism to $\gamma^*\theta_x$ is represented by
\begin{equation*}
\exp(-u(b(s),\alpha(s)))\circ\widehat\Xi_s.
\end{equation*}
In local holomorphic frames, its matrix extends holomorphically in the polydisc parameter with values in $C^1$ on compact fiberwise coordinate discs. By the gauge equations, the matrix is holomorphic on every fiber; Lemma~\ref{chlocal:lem:fiberwise-holomorphicity} then proves its holomorphicity on the reduced total space of the curve family. Its pointwise inverse is holomorphic as well. At $s_0$, the exponential factor is the identity and $\widehat\Xi_{s_0}$ is the prescribed identification, so the isomorphism induces the chosen identification of the central fiber with $x$. Independence of the Sobolev index and $H_x$-equivariance follow from Lemma~\ref{chlocal:lem:kuranishi-regularity}.
\end{proof}

All harmonic spaces and projections used above are those at $x$, so the construction also applies when the cohomology dimensions of nearby fibers vary.

\subsection{Quadratic obstructions and complex gauge transformations}\label{chlocal:sec:quadratic}\label{chlocal:sec:gauge-transformations}

The harmonic deformation space carries the symplectic structure used to express the quadratic obstruction.

\begin{lem}\label{chlocal:prop:symplectic-representation}\label{chlocal:lem:harmonic-duality}\label{chlocal:prop:moment-map}
The space $V_x$, represented by harmonic forms, carries the $H_x$-invariant symplectic form
\begin{equation}\label{chlocal:eq:symplectic-pairing}
\Omega_{\C}((a,\psi),(a',\psi'))=\int_X\tr(a\wedge\psi'-a'\wedge\psi).
\end{equation}
For $\mathfrak h_x=\Lie H_x=H^0$, the trace pairing induces an $H_x$-equivariant isomorphism $H^2\simeq\mathfrak h_x^*$. For $v=(a,\psi)\in V_x$, the left action
\begin{equation*}
\sigma\cdot v=\sigma v\sigma^{-1},
\qquad \sigma\in H_x,
\end{equation*}
has complex moment map
\begin{equation*}
\begin{aligned}
\langle\nu_{x,\C}(v),\xi\rangle
&=\int_X\tr\!\left(\xi\,\tfrac12\mathsf H_2[v,v]\right)\\
&=\tfrac12\Omega_{\C}([\xi,v],v)
=\int_X\tr(\xi[a,\psi]).
\end{aligned}
\end{equation*}
The moment-map identity is
\begin{equation*}
d\langle\nu_{x,\C},\xi\rangle
=\iota_{[\xi,v]}\Omega_{\C}.
\end{equation*}
\end{lem}

The right action in~\eqref{chlocal:eq:right-gauge} is related to this left action by $v^g=g^{-1}\mathbin{\cdot}v$.
The form $\Omega_{\C}$ and the moment map $\nu_{x,\C}$ are the harmonic pairing and quadratic moment map of~\cite[Theorem~5.3(1)--(2) and the preceding discussion]{Fan2020Construction}. They restrict to $\End_0E$ through the parallel orthogonal splitting $\End E=\End_0E\oplus\C\Id_E$. The duality is Serre duality; for $t\in A^{1,1}(\End_0E)$, integration against $\xi\in H^0$ annihilates $\im D_1$, so
\begin{equation*}
\int_X\tr(\xi\mathsf H_2t)=\int_X\tr(\xi t).
\end{equation*}
Cyclicity of the trace verifies the moment-map identity.

\begin{thm}
\label{chlocal:thm:quadratic-product}
On the domain $\Omega=B\times U$ of Proposition~\ref{chlocal:prop:kuranishi-reduction},
\begin{equation}\label{chlocal:eq:exact-obstruction}
\kappa(b,v)=\tfrac12\mathsf H_2[v,v].
\end{equation}
Consequently, for $Q_x=(\nu_{x,\C}^{-1}(0))_{\red}$, the reduced zero locus is
\begin{equation*}
Z_x=B\times(Q_x\cap U)
\end{equation*}
as a reduced analytic subspace of $B\times U$, with its projection to $B$ and the restricted $H_x$-action.
\end{thm}

\begin{proof}
We apply the cancellation argument from~\cite[Theorem~5.1 and Corollary~5.2]{Fan2020Construction} to the relative equation. We write $A=v+w$. Since $w\in\im D_1^*$, the degree-two K\"ahler identity gives
\begin{equation*}
w=D_1^*G_2D_1w=D'f,\qquad f=i\Lambda G_2D_1w.
\end{equation*}
Here $w\in \W{k}$ and $f\in \W{k+1}$, and both are smooth by Lemma~\ref{chlocal:lem:kuranishi-regularity}. The equality of harmonic Laplacians implies $D'v=0$ and $\mathsf H_2D'=0$. Since $D'$ is a graded derivation and $(D')^2=0$,
\begin{equation*}
D'[v,f]=-[v,w],\qquad D'[f,w]=[w,w].
\end{equation*}
The bracket of degree-one forms is symmetric, hence
\begin{equation*}
\mathsf H_2[A,A]=\mathsf H_2[v,v].
\end{equation*}
The relative equation is
\begin{equation*}
F(b,A)=D_1A+\partial_h\bigl(\iota_{\mu(b)}(\varphi+\psi)\bigr)+\tfrac12[A,A].
\end{equation*}
The contracted form has type $(0,1)$, so its bracket with $\varphi^{*h}$ vanishes on a curve. Therefore the middle term equals $D'(\iota_{\mu(b)}(\varphi+\psi))$ and has zero harmonic projection. Projection of this equation proves~\eqref{chlocal:eq:exact-obstruction} on $\Omega$.

By the trace pairing, $\kappa(b,v)=0$ if and only if $\nu_{x,\C}(v)=0$; this condition is independent of $b$. Since a product with a complex manifold preserves reducedness, the reduced zero locus is $B\times(Q_x\cap U)$. Its projection to $B$ is $\operatorname{pr}_B$. Equivariance follows from Lemma~\ref{chlocal:lem:kuranishi-regularity}, with the holomorphic family isomorphisms of Theorem~\ref{chlocal:thm:relative-family}.
\end{proof}

We next compare two representatives in Coulomb gauge related by a complex gauge transformation. For this purpose, we extend $D_0$ to $D''_0f=(\bar\partial_{\End}f,[\varphi,f])$ on all endomorphisms. Its kernel $\mathcal A_0=\C\Id_E\oplus H^0$ is a finite-dimensional algebra of parallel endomorphisms that is closed under the $h$-adjoint. Left and right multiplication by its elements commute with $(D''_0)^*$, since they are parallel and commute with $\varphi^{*h}$. Here $(D''_0)^*$ denotes the adjoint on the full bundle $\End E$. The elliptic estimate on the orthogonal complement of the kernel gives
\begin{equation}\label{chlocal:eq:kernel-estimate}
\|g_\perp\|_{1,2}\leq C_0\|D''_0g_\perp\|_2,\qquad g_\perp\perp\mathcal A_0.
\end{equation}
This is the estimate used in the proof of~\cite[Theorem~4.1]{Fan2020OrbitTypes}.

This estimate controls gauge transformations between representatives in Coulomb gauge with sufficiently small $\W{k}$ norms.

\begin{prop}
\label{chlocal:prop:local-gauge-uniqueness}
There exist a neighborhood $B'\subset B$ of zero and $\varepsilon>0$ such that the following holds for every $b\in B'$. Let $\alpha_i=(a_i,\psi_i)$, $i=1,2$, satisfy $D_0^*\alpha_i=0$ and $\|\alpha_i\|_{k,2}<\varepsilon$. If $g\in\mathcal G_{k+1}^{\C}$ satisfies $(b,\alpha_1)^g=(b,\alpha_2)$, then $g\in H_x$ and $\alpha_2=g^{-1}\alpha_1g$. The choices of $B'$ and $\varepsilon$ are independent of $g$.
\end{prop}

\begin{proof}
All projections, norms, and adjoints below are computed from the fixed data at $x$. Since the adjoint preserves the scalar and trace-free summands, $(D''_0)^*\alpha_i=D_0^*\alpha_i=0$.

We multiply the action equations~\eqref{chlocal:eq:right-gauge} by $g$ and obtain
\begin{equation}\label{chlocal:eq:finite-gauge-equation}
D''_0g=
\left(
ga_2-a_1g+\iota_{\mu(b)}\partial_hg,\,
g\psi_2-\psi_1g
\right).
\end{equation}
We decompose $g=g_\parallel+g_\perp$ by the fixed $L^2$-projection onto $\mathcal A_0$. We have $D''_0g=D''_0g_\perp$ and $\partial_hg_\parallel=0$. The multiplication identities and Coulomb gauge conditions give
\begin{equation*}
(D''_0)^*
(g_\parallel a_2-a_1g_\parallel,
g_\parallel\psi_2-\psi_1g_\parallel)
=
g_\parallel(D''_0)^*\alpha_2
-\bigl((D''_0)^*\alpha_1\bigr)g_\parallel=0.
\end{equation*}
After applying $(D''_0)^*$ to~\eqref{chlocal:eq:finite-gauge-equation}, every term involving $g_\parallel$ cancels.

We set
\begin{equation*}
\mathcal N(g_\perp)=
\left(
g_\perp a_2-a_1g_\perp+\iota_{\mu(b)}\partial_hg_\perp,\,
g_\perp\psi_2-\psi_1g_\perp
\right).
\end{equation*}
The remaining equation is
\begin{equation*}
(D''_0)^*D''_0g_\perp=(D''_0)^*\mathcal N(g_\perp).
\end{equation*}
Pairing with $g_\perp$ and integrating by parts yields
\begin{equation}\label{chlocal:eq:gauge-energy}
\|D''_0g_\perp\|_2^2
\leq
\|\mathcal N(g_\perp)\|_2\,
\|D''_0g_\perp\|_2.
\end{equation}
The stated Sobolev orders justify the products and integration by parts.

The estimate~\eqref{chlocal:eq:kernel-estimate} gives $\|g_\perp\|_{1,2}\leq C_0\|D''_0g_\perp\|_2$. Pointwise multiplication and the embedding $\W{k}\hookrightarrow C^0$ give a constant $C_1$ such that
\begin{equation*}
\|\mathcal N(g_\perp)\|_2
\leq
C_1\bigl(
\|\alpha_1\|_{k,2}+\|\alpha_2\|_{k,2}
+\|\mu(b)\|_{C^0}
\bigr)\|g_\perp\|_{1,2}.
\end{equation*}
We choose $\varepsilon$ and $B'$ so that
\begin{equation*}
C_0C_1
\left(2\varepsilon+\sup_{b\in B'}\|\mu(b)\|_{C^0}\right)<1.
\end{equation*}
Since $\mu(0)=0$, such choices exist. Equations~\eqref{chlocal:eq:gauge-energy} and~\eqref{chlocal:eq:kernel-estimate} then give $D''_0g_\perp=0$ and hence $g_\perp=0$.

Hence $g=g_\parallel$ is a smooth Higgs endomorphism of the central fiber. The original gauge transformation is invertible and has determinant one, so $g\in H_x$. Its parallelism reduces the gauge action in~\eqref{chlocal:eq:right-gauge} to conjugation for every $b$.
\end{proof}

\begin{cor}
\label{chlocal:cor:local-stabilizers}
There is a product domain $B\times U\subset B'\times V_x$ on which $\|A(b,v)\|_{k,2}<\varepsilon$, where $B'$ and $\varepsilon$ are as in Proposition~\ref{chlocal:prop:local-gauge-uniqueness}. For $b\in B$ and $v_1,v_2\in Q_x\cap U$, the equality $\theta_x(b,v_1)^g=\theta_x(b,v_2)$ holds exactly when $g\in H_x$ and $v_2=v_1^g$. In particular $\Stab_{\mathcal G_{k+1}^{\C}}\theta_x(b,v)=\Stab_{H_x}(v)$ for $(b,v)\in Z_x$.
\end{cor}

\begin{proof}
By Proposition~\ref{chlocal:prop:local-gauge-uniqueness}, $g\in H_x$ and
\begin{equation*}
A(b,v_2)=A(b,v_1)^g.
\end{equation*}
Applying $\mathsf H_1$, which commutes with $H_x$, yields $v_2=v_1^g$. The converse follows from~\eqref{chlocal:eq:restricted-equivariance}. The case $v_1=v_2$ proves the stabilizer identity; all these transformations are parallel and hence smooth.
\end{proof}

\subsection{Analytic Hilbert quotients}\label{chlocal:sec:quotients}

The stabilizer $H_x$ is a reductive algebraic group, with maximal compact subgroup $K_x=H_x\cap\mathcal G_{k+1}$ and polar decomposition $H_x=K_x\exp(i\Lie K_x)$~\cite[Proposition~3.1]{Fan2020Construction}. Explicitly, a polystable decomposition $(E,\varphi)=\bigoplus_\alpha(E_\alpha,\varphi_\alpha)\otimes\C^{m_\alpha}$, with pairwise nonisomorphic stable factors of ranks $d_\alpha$, gives
\begin{equation*}
H_x=\{(g_\alpha)\in\textstyle\prod_\alpha\GL(m_\alpha,\C):\prod_\alpha\det(g_\alpha)^{d_\alpha}=1\}.
\end{equation*}
Blockwise polar decomposition preserves this character kernel, including all components. Conjugation on $V_x$ and $Q_x$ is algebraic.

We set $Y_x=(\Spec\C[Q_x]^{H_x})^{\mathrm{an}}$ and write $q_x\colon Q_x\to Y_x$. This is the analytic Hilbert quotient for the $H_x$-action: it is a surjective $H_x$-invariant holomorphic map, it is locally Stein, and
\begin{equation*}
\mathcal O_{Y_x}=((q_x)_*\mathcal O_{Q_x})^{H_x}.
\end{equation*}

\begin{thm}
\label{chlocal:lem:saturated-neighborhood}\label{chlocal:thm:saturated-quotient}\label{chlocal:prop:invariant-descent}
After shrinking $B$, there is a $K_x$-invariant ball $U\subset V_x$, centered at the origin, such that $B\times U\subset\Omega$ and
\begin{equation}\label{chlocal:eq:saturation}
H_x(Q_x\cap U)=q_x^{-1}(q_x(Q_x\cap U)),
\end{equation}
and $q_x(Q_x\cap U)$ is open. There is a reduced Stein neighborhood $\mathcal W_x$ of $[0]$ such that
\begin{equation*}
\widetilde U_x:=q_x^{-1}(\mathcal W_x)\subset H_x(Q_x\cap U).
\end{equation*}
The product map
\begin{equation*}
q_B=\id_B\times q_x\colon B\times\widetilde U_x\longrightarrow B\times\mathcal W_x
\end{equation*}
is the analytic Hilbert quotient for the $H_x$-action, where $H_x$ acts trivially on $B$, and is also a topological quotient.

Each fiber contains a unique closed $H_x$-orbit. Two points have the same image exactly when their orbit closures intersect. Invariant holomorphic maps to Hausdorff complex spaces factor uniquely through $q_B$. For every open subset $O\subset B\times\mathcal W_x$, the restriction $q_B^{-1}(O)\to O$ has the same quotient and fiber properties; if $O$ is Stein, then $q_B^{-1}(O)$ is Stein. The germ of the quotient is $(B\times Y_x,(0,[0]))$ over $B$.
\end{thm}

\begin{proof}
We choose a $K_x$-invariant Hermitian norm on $V_x$ and a ball $U=\{\|v\|<\rho\}$ such that $B\times U\subset\Omega$. Its squared norm restricts to a proper strictly plurisubharmonic exhaustion of the closed affine cone $Q_x\subset V_x$. By~\cite[Propositions~3.1.5--3.1.6, Corollary~3.1.7]{HeinznerHuckleberry1999}, the saturation $H_x(Q_x\cap U)$ is $q_x$-saturated and $q_x(Q_x\cap U)$ is open, which gives~\eqref{chlocal:eq:saturation}.

We choose a reduced Stein neighborhood $\mathcal W_x\subset q_x(Q_x\cap U)$ of $[0]$ and set $\widetilde U_x=q_x^{-1}(\mathcal W_x)$. Then $\widetilde U_x\subset H_x(Q_x\cap U)$ by~\eqref{chlocal:eq:saturation}. Since $H_x$ acts trivially on the coordinate space containing $B$,
\begin{equation*}
\C[\C^{\dim B}\times Q_x]^{H_x}
=\C[\C^{\dim B}]\otimes_{\C}\C[Q_x]^{H_x}.
\end{equation*}
Thus the affine quotient of $\C^{\dim B}\times Q_x$ is $\C^{\dim B}\times\Spec\C[Q_x]^{H_x}$. Its analytification is the analytic Hilbert quotient~\cite[Theorem~3.1.4]{HeinznerHuckleberry1999}. Restriction to the inverse image of $B\times\mathcal W_x$ gives $q_B=\id_B\times q_x$ and
\begin{equation*}
\mathcal O_{B\times\mathcal W_x}
=\bigl((q_B)_*\mathcal O_{B\times\widetilde U_x}\bigr)^{H_x}.
\end{equation*}
If $O\subset B\times\mathcal W_x$ is Stein, then
\begin{equation*}
q_B^{-1}(O)\simeq Q_x\times_{Y_x}O
=\{(v,(b,y))\in Q_x\times O:q_x(v)=y\}
\end{equation*}
is a closed analytic subspace of the Stein space $Q_x\times O$, hence is Stein.

By the Kempf--Ness description, $q_B$ is a topological quotient. Each fiber contains a unique closed orbit, and two points lie in the same fiber exactly when their orbit closures intersect~\cite[Proposition~3.1.2, Corollary~3.1.3]{HeinznerHuckleberry1999}. Since $\widetilde U_x$ is saturated, these statements restrict from $Q_x$ to $\widetilde U_x$. An invariant holomorphic map to a Hausdorff complex space is constant on orbit closures and therefore descends continuously through the topological quotient. In a local embedding of its target into a polydisc, each coordinate function descends holomorphically by the invariant direct-image identity. The descended map is therefore holomorphic. These quotient properties, descriptions of orbit closures, and descent statements hold for every restricted map
\begin{equation*}
q_B^{-1}(O)\longrightarrow O,
\qquad O\subset B\times\mathcal W_x\ \text{open}.
\end{equation*}
\end{proof}

We extend the Kuranishi family $H_x$-equivariantly from $B\times(Q_x\cap U)$ to $B\times H_x(Q_x\cap U)$.

\begin{lem}
\label{chlocal:lem:extended-family}
The map $A$ of the Kuranishi family over $B\times(Q_x\cap U)$ has a unique $H_x$-equivariant extension to $B\times H_x(Q_x\cap U)$, characterized by
\begin{equation*}
\widetilde A(b,v^\sigma)=A(b,v)^\sigma,
\qquad v\in Q_x\cap U,\quad \sigma\in H_x.
\end{equation*}
The extension is holomorphic at every finite Sobolev order, defines a holomorphic Higgs family, and satisfies
\begin{equation*}
F(b,\widetilde A)=0,
\qquad D_0^*\widetilde A=0,
\qquad \mathsf H_1\widetilde A(b,v)=v.
\end{equation*}
\end{lem}

\begin{proof}
On each translate $(Q_x\cap U)^\sigma$, we define
\begin{equation*}
\widetilde A(b,\widetilde v)
=A(b,\widetilde v^{\sigma^{-1}})^\sigma.
\end{equation*}
If $v_1^{\sigma_1}=v_2^{\sigma_2}$, then $v_2=v_1^{\sigma_1\sigma_2^{-1}}$, and restricted equivariance gives
\begin{equation*}
A(b,v_2)^{\sigma_2}
=\bigl(A(b,v_1)^{\sigma_1\sigma_2^{-1}}\bigr)^{\sigma_2}
=A(b,v_1)^{\sigma_1}.
\end{equation*}
The definitions therefore agree on overlaps and glue to a holomorphic map at every finite Sobolev order. The families obtained by applying these gauge transformations to $\theta_x$ glue to a holomorphic Higgs family on $B\times H_x(Q_x\cap U)$. Pullback to any reduced analytic base preserves holomorphicity. Gauge covariance gives the equation for $F$, while Lemma~\ref{chlocal:lem:harmonic-equivariance} gives the identities involving $D_0^*$ and $\mathsf H_1$.
Equivariance determines the extension uniquely.
\end{proof}

\begin{cor}\label{chlocal:cor:extended-stabilizers}
For the extended family $\widetilde\theta_x$ on $B\times\widetilde U_x$, the equality $\widetilde\theta_x(b,v_1)^g=\widetilde\theta_x(b,v_2)$ holds exactly when $g\in H_x$ and $v_2=v_1^g$. Thus $\Stab_{\mathcal G_{k+1}^{\C}}\widetilde\theta_x(b,v)=\Stab_{H_x}(v)$.
\end{cor}

\begin{proof}
We write $v_i=\widetilde v_i^{\sigma_i}$ with $\widetilde v_i\in Q_x\cap U$ and $\sigma_i\in H_x$. If $\widetilde\theta_x(b,v_1)^g=\widetilde\theta_x(b,v_2)$, the right action identity~\eqref{chlocal:eq:right-action-composition} gives
\begin{equation*}
\theta_x(b,\widetilde v_1)^{\sigma_1g\sigma_2^{-1}}
=\theta_x(b,\widetilde v_2).
\end{equation*}
Corollary~\ref{chlocal:cor:local-stabilizers} gives $\sigma_1g\sigma_2^{-1}\in H_x$ and $\widetilde v_2=\widetilde v_1^{\sigma_1g\sigma_2^{-1}}$. Thus $g\in H_x$ and, after applying $\sigma_2$ on the right, $v_2=v_1^g$.
The converse follows from the $H_x$-equivariance of the extended family. Taking $v_1=v_2$ gives the stabilizer equality. Every such gauge transformation lies in $H_x$ and is smooth and parallel.
\end{proof}

The next section proves that closed $H_x$-orbits represent polystable Higgs bundles and that the local analytic quotient has the unitary gauge quotient topology.
\section{Polystability and the gauge quotient topology}
\label{sec:tech-chtopology}
\subsection{The real moment map}\label{chtopology:sec:topology}

We work near a point $x$ with the Kuranishi family $\theta_x$ and its equivariant extension $\widetilde\theta_x$ from Section~\ref{sec:tech-chlocal}. We first solve the real Hitchin equation orthogonally to the stabilizer. The remaining finite-dimensional moment map admits a flow whose convergence is uniform in the curve parameter; its limit will induce continuous maps from the local quotient to the unitary gauge quotient. We put $\fk_x=\Lie K_x$, so that $H_x=K_x\exp(i\fk_x)$.

Section~\ref{sec:tech-chlocal} uses the right action $v^g=g^{-1}vg$. For the Hamiltonian formulas, we use the associated left action
\begin{equation*}
g\cdot v:=v^{g^{-1}}=gvg^{-1},
\qquad \xi_V(v)=[\xi,v].
\end{equation*}
The right gauge action on pairs $(\nabla,\varphi)$ satisfies
\begin{equation*}
\bigl((\nabla,\varphi)^g\bigr)^u=(\nabla,\varphi)^{gu}.
\end{equation*}
Using the fixed area form $\omega$ from Section~\ref{sec:tech-chlocal}, we choose the surface metrics
\begin{equation*}
g_{\Sigma,b}(u,w)=\omega(u,J_bw).
\end{equation*}
The restrictions of the Hodge star to zero-forms and to two-forms are independent of $b$; we denote these operators by $*$. Let $\fg_E$ be the bundle of trace-free $h$-skew-Hermitian endomorphisms. We denote by
\begin{equation*}
(\nabla(b,v),\varphi(b,v))
\end{equation*}
the pair determined by $A(b,v)$ and $h$. The real moment map takes values in $\W{k-1}\Omega^2(\fg_E)$ and is given by
\begin{equation}\label{chtopology:eq:real-moment-map}
\boldsymbol\mu_{\R}(\nabla,\varphi)=F_\nabla+[\varphi,\varphi^{*h}].
\end{equation}
Equivariance of $A$ and uniqueness of the unitary connection inducing a fixed Dolbeault operator give
\begin{equation*}
(\nabla(b,u\cdot v),\varphi(b,u\cdot v))
=(\nabla(b,v),\varphi(b,v))^{u^{-1}},
\qquad u\in K_x.
\end{equation*}

At the solution to the Hitchin equations representing $x$, we set $d_1\xi=(d_{\nabla_h}\xi,[\varphi,\xi])$ for $\xi\in\W{k+1}\Omega^0(\fg_E)$, and let $d_2$ be the derivative of $\boldsymbol\mu_{\R}$ there. The moment-map identity and elliptic Hodge theory give
\begin{equation*}
d_2^*=-I_0d_1*,\qquad
\cH_{\R,x}^2:=\ker d_2^*=*\fk_x,
\end{equation*}
where $I_0$ is the complex structure on the tangent space at this solution to the Hitchin equations. All operators preserve the trace-free summand. For every $s\geq0$, elliptic Hodge theory gives an isomorphism
\begin{equation*}
d_2d_2^*:
(\cH_{\R,x}^2)^\perp\cap\W{s+2}\Omega^2(\fg_E)
\xrightarrow{\ \sim\ }
(\cH_{\R,x}^2)^\perp\cap\W{s}\Omega^2(\fg_E);
\end{equation*}
see~\cite[proof of Lemma~3.5]{Fan2020Construction}. We write $\mathsf H_{\R,x}$ for the fixed orthogonal projection onto $\cH_{\R,x}^2$ and reserve $H^2$ for the complex obstruction space defining $Q_x$.

The bounded inverse of $d_2d_2^*$ allows us to solve for the component of the real moment map orthogonal to $\cH_{\R,x}^2$.

\begin{prop}
\label{chtopology:prop:real-moment-reduction}
After replacing $B$ and $U$ by neighborhoods whose closures lie in the original domains, there are $\varepsilon>0$ and a unique smooth map
\begin{equation*}
\beta:B\times U\longrightarrow
(\cH_{\R,x}^2)^\perp\cap\W{k+1}\Omega^2(\fg_E)
\end{equation*}
such that $\beta(0,0)=0$, $\|\beta(b,v)\|_{k+1,2}<\varepsilon$, and
\begin{equation*}
(1-\mathsf H_{\R,x})\boldsymbol\mu_{\R}
\left((\nabla(b,v),\varphi(b,v))^{\exp(-i*\beta(b,v))}\right)=0.
\end{equation*}
We set
\begin{equation*}
\begin{gathered}
g^+(b,v)=\exp(-i*\beta(b,v)),\\
(\nabla^+(b,v),\varphi^+(b,v))
=(\nabla(b,v),\varphi(b,v))^{g^+(b,v)},\\
\nu_{\R}(b,v)=\boldsymbol\mu_{\R}(\nabla^+(b,v),\varphi^+(b,v)).
\end{gathered}
\end{equation*}
Then $\nu_{\R}$ takes values in $\cH_{\R,x}^2$. After this restriction of $B\times U$, the maps $\beta$, $g^+$, $(\nabla^+,\varphi^+)$, and $\nu_{\R}$ are $K_x$-equivariant and smooth with values in the corresponding Sobolev spaces of every finite order. Their germs agree at different Sobolev indices, and $d_v\beta(0,0)=0$.
\end{prop}

\begin{proof}
We apply the parameter-dependent implicit function theorem to
\begin{equation*}
\mathcal L(b,v,\beta)=(1-\mathsf H_{\R,x})
\boldsymbol\mu_{\R}\left((\nabla(b,v),\varphi(b,v))^{\exp(-i*\beta)}\right).
\end{equation*}
Sobolev multiplication and the mapping properties of the gauge action show that
\begin{equation*}
\mathcal L:B\times U\times
\bigl((\cH_{\R,x}^2)^\perp\cap \W{k+1}\Omega^2(\fg_E)\bigr)
\longrightarrow
(\cH_{\R,x}^2)^\perp\cap \W{k-1}\Omega^2(\fg_E)
\end{equation*}
is smooth. Its derivative in the last variable is
\begin{equation*}
D_\beta\mathcal L(0,0,0)\dot\beta
=(1-\mathsf H_{\R,x})d_2(-I_0d_1*\dot\beta)
=d_2d_2^*\dot\beta.
\end{equation*}
The inverse of $d_2d_2^*$ acts from $(\cH_{\R,x}^2)^\perp\cap\W{k-1}$ to $(\cH_{\R,x}^2)^\perp\cap\W{k+1}$. We shrink $B\times U$ so that
\begin{equation*}
\bigl\|(d_2d_2^*)^{-1}(D_\beta\mathcal L-d_2d_2^*)\bigr\|<\tfrac12
\end{equation*}
holds; the linearized inverses are then uniformly bounded. We choose a $K_x$-invariant Hermitian norm on $V_x$ and take $U$ to be a $K_x$-invariant ball. Equivariance of $\mathcal L$ and uniqueness of its zero in the $\varepsilon$-ball give equivariance of $\beta$. The exponent $-i*\beta$ is Hermitian and trace-free, so $g^+$ is positive and has determinant one. The derivative of $(\nabla(b,v),\varphi(b,v))$ with respect to $v$ at $(0,0)$ is the inclusion of harmonic representatives and lies in $\ker d_2$. Differentiating the equation for $\mathcal L$ gives
\begin{equation*}
d_2d_2^*d_v\beta(0,0)=0.
\end{equation*}
Thus $d_v\beta(0,0)=0$, and $d_v(\nabla^+,\varphi^+)(0,0)$ is the inclusion of harmonic representatives.

To prove smoothness on $\Sigma$, we transport the metric by $g^+$ and define
\begin{equation*}
h'(e_1,e_2)=h\bigl((g^+)^{-1}e_1,(g^+)^{-1}e_2\bigr).
\end{equation*}
It satisfies
\begin{equation*}
F_{h'}+[\varphi(b,v),\varphi(b,v)^{*h'}]
=g^+\nu_{\R}(b,v)(g^+)^{-1}.
\end{equation*}
The form $\nu_{\R}(b,v)$ has smooth coefficients because it belongs to the fixed finite-dimensional space $\cH_{\R,x}^2$, and the coefficients of the original family are smooth by Section~\ref{sec:tech-chlocal}. In a local holomorphic frame, the curvature formula $\bar\partial(H^{-1}\partial H)$ expresses the second derivatives of the matrix $H$ of $h'$ in terms of its first derivatives and smooth functions of $H$ and the fixed coefficients. Indeed, $g^+=(h^{-1}h')^{-1/2}$, so the right-hand side contributes only terms of order zero in $H$. Starting from $H\in\W{k+1}$, elliptic regularity and iteration give $C^\infty$ regularity on $\Sigma$ for $h'$, $g^+$, and $\beta$.

At each solution, $D_\beta\mathcal L$ is elliptic of order two modulo a smooth finite-rank operator. Its inverse from order $k-1$ to order $k+1$, together with elliptic regularity, induces bounded inverses at every higher Sobolev order. Fix $(b_0,v_0)\in B\times U$. Since $\beta(b_0,v_0)$ is smooth by the preceding paragraph, the implicit function theorem at every higher order applies near this solution. Uniqueness of the solution in $\W{k+1}$ identifies each resulting local map with $\beta$. These neighborhoods cover the already chosen product $B\times U$, so no further restriction depending on the Sobolev order is needed. The assertions for $g^+$, $(\nabla^+,\varphi^+)$, and $\nu_{\R}$ follow.
\end{proof}

We denote by $\Omega_{I,b}$ the $L^2$ K\"ahler form on the space of pairs $(\nabla,\varphi)$, normalized as in~\cite[Section~3.1]{Fan2020Construction}. Its pullback along $v\mapsto(\nabla^+(b,v),\varphi^+(b,v))$ is denoted by $\omega_b$.
Equivariance and the moment-map identity give
\begin{equation*}
d_v\langle\nu_{\R},\xi\rangle
=\iota_{[v,\xi]}\omega_b,
\qquad
\langle\upsilon,\xi\rangle=\int_\Sigma\tr(\xi\upsilon).
\end{equation*}
The forms $\omega_b$ are closed. Since $\omega_0(0)$ is the K\"ahler form induced by the $L^2$ inner product on $H^1$, they are nondegenerate on a product neighborhood of $(0,0)$.

\begin{lem}
\label{chtopology:lem:zero-section}
For every $b\in B$, one has $\nu_{\R}(b,0)=0$. For $(b,v)\in Z_x$, the pair $(\nabla^+(b,v),\varphi^+(b,v))$ satisfies the Hitchin equations exactly when $\nu_{\R}(b,v)=0$.
\end{lem}

\begin{proof}
By equivariance, $(\nabla^+(b,0),\varphi^+(b,0))=(\nabla_h+\delta_b,\varphi_b^+)$ is fixed by $K_x$. For $\xi\in\fk_x$, the identity $\nabla_h\xi=0$ and $K_x$-invariance give $[\delta_b,\xi]=[\varphi_b^+,\xi]=[(\varphi_b^+)^{*h},\xi]=0$, and hence
\begin{equation*}
\tr\bigl(\xi(F_{\nabla_h+\delta_b}-F_{\nabla_h})\bigr)
=d\tr(\xi\delta_b),\qquad
\tr(\xi[\varphi_b^+,(\varphi_b^+)^{*h}])=0.
\end{equation*}
The quadratic curvature term pairs to zero because
\begin{equation*}
\tr(\xi[\delta_{b,1},\delta_{b,2}])
=\tr([\xi,\delta_{b,1}]\delta_{b,2})=0.
\end{equation*}
The Hitchin equation at $x$ and Stokes' theorem imply $\int_\Sigma\tr(\xi\nu_{\R}(b,0))=0$. The pairing with $\fk_x$ is nondegenerate on $\cH_{\R,x}^2=*\fk_x$, so $\nu_{\R}(b,0)=0$. Since $B\times\{0\}\subset Z_x$ and complex gauge transformations preserve the holomorphic Higgs equation, $(\nabla^+(b,0),\varphi^+(b,0))$ is a solution to the Hitchin equations.
\end{proof}

\subsection{Moment-map flow and polystability}

We denote the real $L^2$ inner product on $V_x$ by $\langle\ ,\ \rangle_0$, let $I$ be multiplication by $i$, and choose an invariant inner product identifying $\fk_x^*$ with $\fk_x$ by $\#$. For the left $K_x$-action, the finite-dimensional moment map $\ell_b:U\to\fk_x^*$ and its squared norm are
\begin{equation*}
\langle\ell_b(v),\xi\rangle=-\int_\Sigma\tr(\xi\nu_{\R}(b,v)),
\qquad f_b(v)=\tfrac12\|\ell_b(v)\|^2.
\end{equation*}
Then $d\langle\ell_b,\xi\rangle=\iota_{\xi_V}\omega_b$ and $\ell_b(0)=0$. The quadratic moment map for the linear $K_x$-action on $(V_x,\omega_{\mathrm{lin}})$ is denoted by $\ell_{\mathrm{lin}}$ and is given by
\begin{equation*}
\omega_{\mathrm{lin}}(u,w)=\langle Iu,w\rangle_0,\qquad
\langle\ell_{\mathrm{lin}}(v),\xi\rangle
=\tfrac12\omega_{\mathrm{lin}}(\xi_V(v),v).
\end{equation*}
From now on, $B$ denotes a sufficiently small neighborhood of the origin whose closure lies in the parameter domain of Proposition~\ref{chtopology:prop:real-moment-reduction}. In particular, $\omega_b(0)$ is uniformly close to $\omega_{\mathrm{lin}}$ for $b\in\overline B$.

\begin{lem}
\label{chtopology:lem:moment-estimates}
There are $\rho_2>0$ and positive constants $c_0,C_\omega,C$, independent of $b$, such that the $K_x$-invariant ball
\begin{equation*}
U_2=\{v\in V_x:\|v\|<\rho_2\}
\end{equation*}
has closure in $U$ and, for all $b\in B$, $v\in U_2$, and $u\in V_x$,
\begin{equation}\label{chtopology:eq:moment-estimates}
\omega_b(v)(u,Iu)\geq c_0\|u\|^2,\qquad
\|\omega_b(v)\|\leq C_\omega,\qquad
f_b(v)^{3/4}\leq C\|(\ell_b(v)^\#)_V(v)\|.
\end{equation}
\end{lem}

\begin{proof}
The first two bounds follow from continuity and compactness on a ball with closure in $U$. On this ball, we apply the equivariant Moser argument to
\begin{equation*}
\omega_{b,t}=\omega_{\mathrm{lin}}+t(\omega_b-\omega_{\mathrm{lin}}),\qquad 0\leq t\leq1.
\end{equation*}
These forms uniformly tame $I$ and therefore have uniformly bounded inverses. We define the invariant radial primitive $\vartheta_b$ by
\begin{equation*}
\vartheta_b(v)(w)
=\int_0^1 t(\omega_b-\omega_{\mathrm{lin}})_{tv}(v,w)\,dt.
\end{equation*}
Then $d\vartheta_b=\omega_b-\omega_{\mathrm{lin}}$, $\vartheta_b(0)=0$, and $\|\vartheta_b(v)\|\leq C\|v\|$. We denote by $\Psi_{b,t}$ the Moser isotopy determined by
\begin{equation*}
\iota_{\dot\Psi_{b,t}\circ\Psi_{b,t}^{-1}}\omega_{b,t}=-\vartheta_b,
\qquad \Psi_{b,0}=\id.
\end{equation*}
The vector field defining this isotopy vanishes at the origin and has uniformly bounded first derivatives. Its forward and backward trajectories starting sufficiently near the origin therefore remain in the chosen ball for $0\leq t\leq1$. We choose $\rho_2$ so that the isotopy and its inverse are defined on a neighborhood of $\overline{U_2}$ throughout this interval. Their first derivatives are uniformly bounded in $b$ and $t$. The isotopy is $K_x$-equivariant, fixes the origin, and satisfies $\Psi_{b,1}^*\omega_b=\omega_{\mathrm{lin}}$.

By definition, $d\langle\ell_{\mathrm{lin}},\xi\rangle=\iota_{\xi_V}\omega_{\mathrm{lin}}$. Equivariance shows that $\ell_b\circ\Psi_{b,1}$ is another moment map for the same form. Their difference is constant and vanishes at zero by Lemma~\ref{chtopology:lem:zero-section}. Thus
\begin{equation*}
\ell_b\circ\Psi_{b,1}=\ell_{\mathrm{lin}}.
\end{equation*}

By~\cite[Theorem~4.7]{Fisher2012}, the fixed polynomial $\|\ell_{\mathrm{lin}}\|^2/2$ satisfies the gradient inequality with exponent $3/4$. The uniform derivative bounds transfer it to $f_b^{3/4}\leq C\|df_b\|$. Since
\begin{equation*}
df_b(v)[w]
=\omega_b(v)\bigl((\ell_b(v)^\#)_V(v),w\bigr),
\end{equation*}
the upper bound for $\omega_b$ gives the conclusion after decreasing $\rho_2$.
\end{proof}

The exponent $3/4$ bounds the remaining trajectory length and makes the flow limit continuous in the curve parameter.

\begin{prop}
\label{chtopology:prop:flow}\label{chtopology:cor:flow-orbits}
There are radii $0<\rho_1<\rho_*<\rho_2$ such that, for
\begin{equation*}
U_1=\{v\in V_x:\|v\|<\rho_1\},
\end{equation*}
the solution of
\begin{equation}\label{chtopology:eq:flow}
\dot v_b=-I(\ell_b(v_b)^\#)_V(v_b),\qquad v_b(0)=v\in U_1,
\end{equation}
exists for all $t\geq0$ and remains in $\{\|v\|\leq\rho_*\}\subset U_2$. The limit
\begin{equation*}
R(b,v)=\lim_{t\to\infty}v_b(t)
\end{equation*}
defines a continuous $K_x$-equivariant map $R:B\times U_1\to U_2$. One has $\ell_b(R(b,v))=0$ and, uniformly in $b$,
\begin{equation}\label{chtopology:eq:flow-decay}
\|v_b(t)-R(b,v)\|\leq C(1+t)^{-1/2},\qquad
\|R(b,v)-v\|\leq C f_b(v)^{1/4}.
\end{equation}
Moreover,
\begin{equation*}
(b,v)\longmapsto
(\nabla^+(b,R(b,v)),\varphi^+(b,R(b,v)))
\end{equation*}
is continuous in the $C^\infty$ topology. For every finite $t$, one has $v_b(t)\in H_x\cdot v$. Consequently the flow preserves $Q_x$, its limit belongs to $\overline{H_x\cdot v}$, and $q_x(R(b,v))=q_x(v)$ for $v\in Q_x\cap U_1$.
\end{prop}

\begin{proof}
The moment-map identity and the first estimate in Lemma~\ref{chtopology:lem:moment-estimates} give
\begin{equation*}
\begin{aligned}
-\dot f_b
&=\omega_b\left((\ell_b(v_b)^\#)_V(v_b),
I(\ell_b(v_b)^\#)_V(v_b)\right)
\geq c_0\|(\ell_b(v_b)^\#)_V(v_b)\|^2,\\
\|\dot v_b\|&=\|(\ell_b(v_b)^\#)_V(v_b)\|.
\end{aligned}
\end{equation*}
As long as $f_b(v_b(t))>0$, Lemma~\ref{chtopology:lem:moment-estimates} and the preceding identities imply
\begin{equation*}
-\frac{d}{dt}f_b^{1/4}
\geq\frac{c_0}{4}f_b^{-3/4}
\|(\ell_b(v_b)^\#)_V(v_b)\|^2
\geq c_1\|\dot v_b\|,
\qquad -\dot f_b\geq a_1 f_b^{3/2}.
\end{equation*}
The length between two finite times is bounded by a constant times the decrease in $f_b^{1/4}$. If $f_b(v_b(t_0))=0$, then the trajectory is constant for $t\geq t_0$. The moment-map identity yields $d\ell_b(0)=0$; since $\ell_b(0)=0$, smooth dependence on $b$ gives $\|\ell_b(v)\|\leq C\|v\|^2$ uniformly near the origin. We may therefore choose $\rho_1$ so that $\|v\|+C f_b(v)^{1/4}<\rho_*$ for $v\in U_1$, uniformly in $b$. The first inequality keeps every trajectory in the closed ball of radius $\rho_*$ and ensures global existence. If $f_b(v)>0$, integration of the second inequality gives rise to
\begin{equation*}
f_b(v_b(t))\leq
\bigl(f_b(v)^{-1/2}+a_1t/2\bigr)^{-2}
\leq C(1+t)^{-2},
\end{equation*}
where the constants are uniform on $B\times U_1$. The remaining trajectory length is at most $C f_b(v_b(t))^{1/4}$. Thus the trajectories are uniformly Cauchy, their limits lie in $\ell_b^{-1}(0)$, and the stated distance bounds follow. Finite-time solutions depend smoothly on $(b,v)$, so the uniform tail estimate implies joint continuity of their limits. Equivariance follows from uniqueness.

On each finite interval, we solve $g'g^{-1}=-i\ell_b(v_b)^\#$ in $H_x$ with $g(0)=1$. Uniqueness gives $v_b(t)=g(t)\cdot v$. This proves that the trajectory remains in $H_x\cdot v$ and that $Q_x$ is preserved. Continuity of invariant polynomials gives $q_x(R(b,v))=q_x(v)$.

Proposition~\ref{chtopology:prop:real-moment-reduction} bounds the derivative of $(\nabla^+,\varphi^+)$ with respect to $v$ in every finite Sobolev norm on $B\times\{\|v\|\leq\rho_*\}$. The mean value inequality consequently gives, for every $s\geq k$,
\begin{equation*}
\|(\nabla^+,\varphi^+)(b,v_b(t))
-(\nabla^+,\varphi^+)(b,R(b,v))\|_{s,2}
\leq C_s(1+t)^{-1/2}.
\end{equation*}
The map $(b,v)\mapsto(\nabla^+,\varphi^+)(b,R(b,v))$ is therefore continuous with respect to every $C^m$-seminorm, and hence in the $C^\infty$ topology.
\end{proof}

For a fixed curve, a degree-zero Higgs bundle is polystable precisely when its complex gauge orbit contains a solution to the Hitchin equations, unique up to the action of the unitary gauge group~\cite[\S1, \S3.1]{Fan2020Construction}. The closure of a semistable orbit contains a unique polystable orbit~\cite[Lemma~3.7]{Fan2020Construction}.

These statements apply with the specified determinant trivialization and trace-free Higgs field. If a complex gauge transformation $g$ takes one such pair to another, then $\det g$ is a nonzero constant on the compact curve. Multiplying $g$ by a scalar $c$ with $c^r=(\det g)^{-1}$ produces a transformation of determinant one with the same action. The determinant metric of a solution to the Hitchin equations is constant and can be normalized by rescaling. The following degeneration identifies the polystable orbit in the closure while retaining the determinant trivialization.

\begin{lem}\label{chtopology:lem:seshadri-degeneration}
Let $(F,\varphi_F)$ be a smooth semistable $\SL_r(\C)$-Higgs bundle on a compact Riemann surface, and fix a Seshadri filtration. After a smooth identification of $F$ with its graded bundle $F^{\mathrm{gr}}$ preserving the determinant trivialization induced by the filtration, there is a family of smooth complex gauge transformations $g_t$, $0<t\leq1$, of determinant one such that
\begin{equation}\label{chtopology:eq:seshadri-degeneration}
(\bar\partial_F,\varphi_F)^{g_t}
\longrightarrow(\bar\partial_{F^{\mathrm{gr}}},\varphi_{F^{\mathrm{gr}}})
\qquad\text{as }t\longrightarrow0
\end{equation}
in the $C^\infty$ topology and in every finite Sobolev norm.
\end{lem}

\begin{proof}
The Seshadri filtration has stable degree-zero Higgs quotients~\cite[Proposition~5.1]{Wilkin08}. Choose a smooth splitting compatible with its induced determinant identification, and write the graded summands as $E_i^{\mathrm{gr}}$, of ranks $r_i$. In this splitting, the Dolbeault operator and Higgs field are block upper triangular. Define
\begin{equation}\label{chtopology:eq:degeneration-weights}
g_t|_{E_i^{\mathrm{gr}}}
=t^{\,r i-\sum_j j r_j}\Id_{E_i^{\mathrm{gr}}}.
\end{equation}
Since $\sum_i r_i=r$, the weighted sum of the exponents is zero, and hence $\det g_t=1$. Under the right gauge action, an off-diagonal block from $E_j^{\mathrm{gr}}$ to $E_i^{\mathrm{gr}}$, with $i<j$, is multiplied by $t^{r(j-i)}$. The diagonal blocks remain fixed, because $g_t$ is scalar and constant on each graded summand. All off-diagonal coefficients and their derivatives therefore tend to zero, proving~\eqref{chtopology:eq:seshadri-degeneration}.
\end{proof}

To apply these fixed-curve results, we establish semistability for every fiber of $\theta_x$ and $\widetilde\theta_x$.

\begin{lem}
\label{chtopology:lem:semistability}
After shrinking $B$, $U$, and $\mathcal W_x$, every Higgs bundle in the family $\theta_x$ over $Z_x$ is semistable. The same holds for $\widetilde\theta_x$ over $B\times\widetilde U_x$.
\end{lem}

\begin{proof}
For an invariant holomorphic subbundle $F$ with orthogonal projection $\Pi_F$, the Higgs Chern--Weil inequality~\cite[proof of Lemma~5.6, equation~(83)]{Wilkin08} gives
\begin{equation*}
\begin{aligned}
\deg F&\leq\frac1{2\pi}\int_\Sigma
\tr\!\left(\Pi_F i*\boldsymbol\mu_{\R}(\nabla(b,v),\varphi(b,v))\right)\omega\\
&\leq\frac{r\Vol(\Sigma)}{2\pi}
\|i*\boldsymbol\mu_{\R}(\nabla(b,v),\varphi(b,v))\|_{L^\infty,\mathrm{op}}.
\end{aligned}
\end{equation*}
Here $i*\boldsymbol\mu_{\R}$ is Hermitian and $\Pi_F$ has rank at most $r$. The last norm is continuous by the embedding $\W{k-1}\subset C^0$ and vanishes at $(0,0)$. We shrink $B\times U$ so that
\begin{equation*}
\frac{r\Vol(\Sigma)}{2\pi}
\|i*\boldsymbol\mu_{\R}(\nabla(b,v),\varphi(b,v))\|_{L^\infty,\mathrm{op}}<1.
\end{equation*}
We choose the radii in Proposition~\ref{chtopology:prop:flow} for these restricted domains. Since $\deg F$ is an integer, it follows that $\deg F\leq0$. The saturation of a Higgs-invariant coherent subsheaf is a Higgs-invariant subbundle on a smooth curve and has no smaller degree. Thus every fiber of $\theta_x$ over $Z_x$ is semistable. Semistability is invariant under complex gauge transformations, and $\widetilde U_x\subset H_x(Q_x\cap U)$, so the conclusion extends to $\widetilde\theta_x$ over $B\times\widetilde U_x$.
\end{proof}

We denote by $\operatorname{Sol}_{B,k}\subset\mathcal C_{B,k}$ the solution space of the Hitchin equations with its $\W{k}$-norm topology and by $\operatorname{Sol}_{B,\infty}$ its smooth counterpart with the $C^\infty$ topology. We write $\mathcal G_\infty$ for the smooth unitary gauge group of determinant one. Both solution spaces carry their natural projections to $B$, and their unitary orbit spaces carry the quotient topologies.

The moment-map flow now gives continuous maps from the local analytic quotient to both unitary quotients.

\begin{prop}
\label{chtopology:prop:graded-classes}\label{chtopology:prop:smooth-lift}
Assume that $\mathcal W_x$ has been shrunk so that $\widetilde U_x=q_x^{-1}(\mathcal W_x)\subset H_x(Q_x\cap U_1)$. Then $\widetilde\theta_x(b,v)$ is polystable exactly when $H_x\cdot v$ is closed in $\widetilde U_x$. The Higgs bundle associated with the closed orbit in $q_x^{-1}(q_x(v))$ represents the Seshadri graded object of $\widetilde\theta_x(b,v)$.

There are compatible continuous injective maps
\begin{equation*}
\chi_x:B\times\mathcal W_x\to\operatorname{Sol}_{B,k}/\mathcal G_{k+1},\qquad
\chi_{x,\infty}:B\times\mathcal W_x\to\operatorname{Sol}_{B,\infty}/\mathcal G_\infty,
\end{equation*}
given, for $v\in\widetilde U_x\cap U_1$, by the class of
\begin{equation*}
(\nabla^+(b,R(b,v)),\varphi^+(b,R(b,v))).
\end{equation*}
\end{prop}

\begin{proof}
We choose $\mathcal W_x$ by Theorem~\ref{chlocal:lem:saturated-neighborhood}. Every orbit in $\widetilde U_x$ meets $U_1$, and
\begin{equation*}
\overline{H_x\cdot v}\subset q_x^{-1}(q_x(v))\subset\widetilde U_x.
\end{equation*}
Thus an orbit in $\widetilde U_x$ is closed there exactly when it is closed in $V_x$. Semistability follows from Lemma~\ref{chtopology:lem:semistability}.

We denote by $\mathcal O_{\mathrm{cl}}$ the closed orbit over a point of $\mathcal W_x$ and choose $v_0\in\mathcal O_{\mathrm{cl}}\cap U_1$. Then $z=R(b,v_0)\in\mathcal O_{\mathrm{cl}}$ satisfies $\nu_{\R}(b,z)=0$, so $(\nabla^+(b,z),\varphi^+(b,z))$ is a solution to the Hitchin equations. Equivariance shows that every fiber of $\widetilde\theta_x$ over $\mathcal O_{\mathrm{cl}}$ is polystable.

For $v\in\widetilde U_x$ with $\mathcal O_{\mathrm{cl}}\subset\overline{H_x\cdot v}$, we choose a sequence $g_j\in H_x$ such that $g_j\cdot v\to z$. Equivariance gives
\begin{equation*}
\widetilde\theta_x(b,g_j\cdot v)\longrightarrow
\widetilde\theta_x(b,z)
\end{equation*}
in the $\W{k}$-norm topology, and every term is equivalent under the complex gauge group to $\widetilde\theta_x(b,v)$. Applying $g^+(b,z)$ shows that the solution to the Hitchin equations $(\nabla^+(b,z),\varphi^+(b,z))$ lies in the closure of the same orbit under the complex gauge group of determinant one.

If $v\in U_1$, the moment-map flow of Proposition~\ref{chtopology:prop:flow} gives another solution to the Hitchin equations at $R(b,v)$ in this closure. The two solutions lie in the same polystable orbit, so a complex gauge transformation $g$ of determinant one takes the solution over $z$ to the solution over $R(b,v)$. The composite gauge transformation
\begin{equation*}
g^+(b,z)\,g\,g^+(b,R(b,v))^{-1}
\end{equation*}
takes $\widetilde\theta_x(b,z)$ to $\widetilde\theta_x(b,R(b,v))$. It follows from Corollary~\ref{chlocal:cor:extended-stabilizers} that $R(b,v)\in\mathcal O_{\mathrm{cl}}$. Conversely, if $\widetilde\theta_x(b,v)$ is polystable, uniqueness of the polystable orbit in its closure identifies it with the bundle over $z$, so the same corollary implies $v\in\mathcal O_{\mathrm{cl}}$. By Lemma~\ref{chtopology:lem:seshadri-degeneration}, a representative of the Seshadri graded object lies in this orbit closure in the $C^\infty$ topology, and hence in the $\W{k}$ topology. Since the graded object is polystable, uniqueness identifies it with $\widetilde\theta_x(b,z)$, with its induced determinant trivialization.

The restricted map
\begin{equation*}
q_{U_1}:=q_B|_{B\times(\widetilde U_x\cap U_1)}
\colon B\times(\widetilde U_x\cap U_1)\longrightarrow B\times\mathcal W_x,
\qquad (b,v)\mapsto(b,q_x(v)),
\end{equation*}
is surjective because every orbit in $\widetilde U_x$ meets $U_1$. For any subset $\mathcal E\subset B\times\mathcal W_x$,
\begin{equation*}
q_B^{-1}(\mathcal E)
=\bigcup_{\sigma\in H_x}\sigma\cdot q_{U_1}^{-1}(\mathcal E).
\end{equation*}
If $q_{U_1}^{-1}(\mathcal E)$ is open in $B\times(\widetilde U_x\cap U_1)$, this union is open in $B\times\widetilde U_x$. The quotient property of Theorem~\ref{chlocal:thm:saturated-quotient} then makes $\mathcal E$ open. Thus $q_{U_1}$ is a quotient map.

To $(b,v)$ in the domain of $q_{U_1}$, we assign the unitary gauge equivalence class of
\begin{equation*}
(\nabla^+(b,R(b,v)),\varphi^+(b,R(b,v))).
\end{equation*}
Proposition~\ref{chtopology:prop:flow} makes this assignment continuous into $\operatorname{Sol}_{B,\infty}$ before passing to the quotient, and hence into both unitary quotients. If $(b,v_1)$ and $(b,v_2)$ lie in one fiber of $q_{U_1}$, the orbit argument above shows that their limits $z_i=R(b,v_i)$ lie on the same closed $H_x$-orbit. By equivariance, $\widetilde\theta_x(b,z_1)$ and $\widetilde\theta_x(b,z_2)$ are related by a smooth complex gauge transformation, and the transformations $g^+(b,z_i)$ give the corresponding equivalence between the two solutions to the Hitchin equations.

The uniqueness assertion in the Hitchin--Kobayashi correspondence gives a unitary gauge transformation $u$ between these two smooth solutions. In local frames with connection forms $\Gamma_i$, the equation
\begin{equation*}
du=u\Gamma_2-\Gamma_1u
\end{equation*}
raises the Sobolev regularity of $u$ by one derivative at each step; hence $u$ is smooth. Consequently both maps to the unitary quotients are constant on the fibers of $q_{U_1}$ and descend continuously to $\chi_x$ and $\chi_{x,\infty}$. Their compatibility follows after pullback by the surjective map $q_{U_1}$.

If two values of $\chi_x$ agree, their $B$-coordinates agree and the two solutions to the Hitchin equations differ by a unitary gauge transformation. Composing with $g^+(b,z_1)$ and $g^+(b,z_2)^{-1}$ gives an isomorphism between the corresponding fibers of $\widetilde\theta_x$. Corollary~\ref{chlocal:cor:extended-stabilizers} then shows that the two closed $H_x$-orbits coincide. The same argument applies to $\chi_{x,\infty}$, proving injectivity of both maps.
\end{proof}

\subsection{Quotient topology}

We denote by $\cP_{B,k}^{\mathrm{ps}}$ and $\cP_{B,\infty}^{\mathrm{ps}}$ the spaces of relative polystable Higgs pairs with fixed determinant and trace-free Higgs field in the $\W{k}$ and $C^\infty$ categories, respectively. We write $\mathcal G_\infty^{\C}$ for the smooth complex gauge group of determinant one. The orbit spaces
\begin{equation*}
\cP_{B,k}^{\mathrm{ps}}/\mathcal G_{k+1}^{\C},
\qquad
\cP_{B,\infty}^{\mathrm{ps}}/\mathcal G_\infty^{\C}
\end{equation*}
carry their quotient topologies.

We first compare smooth and Sobolev representatives and their gauge equivalences.

\begin{lem}
\label{chtopology:lem:sobolev-regularity}
Every $\W{k}$ Higgs pair is equivalent under the complex gauge group to a smooth Higgs pair, and every $\W{k}$ solution to the Hitchin equations is equivalent under the unitary gauge group to a smooth solution to the Hitchin equations. Every complex gauge transformation of class $\W{k+1}$ taking one smooth Higgs pair to another is smooth. All gauge transformations may be chosen to have determinant one.
\end{lem}

\begin{proof}
We apply~\cite[Lemma~2.9]{Trautwein2017} to the principal $\mathrm{SU}(r)$ bundle. The resulting $\W{2}$ complex gauge transformation takes values in $\SL_r(\C)$ and transforms the Dolbeault operator into a smooth one. Elliptic regularity for the gauge equation places this transformation in $\W{k+1}$, after which $\bar\partial_\nabla\varphi=0$ implies smoothness of the transformed Higgs field. The same gauge equation proves that a complex gauge transformation of Sobolev class between smooth Higgs pairs is smooth; compare~\cite[Lemma~14.9]{AB83}.

We fix a solution to the Hitchin equations $y=(\nabla,\varphi)$ and a gauge transformation $g\in\mathcal G_{k+1}^{\C}$ for which the Higgs pair underlying $y^g$ is smooth. The metric $H=g^*h$ satisfies the Hitchin equation on that smooth Higgs bundle. By the elliptic regularity argument in Proposition~\ref{chtopology:prop:real-moment-reduction}, $H$ and the positive endomorphism $h^{-1}H$ are smooth; the latter has determinant one. Then
\begin{equation*}
u=g(h^{-1}H)^{-1/2}
\end{equation*}
is a unitary gauge transformation of determinant one. The transformed pair
\begin{equation*}
y^u=(y^g)^{(h^{-1}H)^{-1/2}}
\end{equation*}
is a smooth solution to the Hitchin equations.
\end{proof}

The maps $\chi_x$ and $\chi_{x,\infty}$ give local coordinates on $\operatorname{Sol}_{B,k}/\mathcal G_{k+1}$ and $\operatorname{Sol}_{B,\infty}/\mathcal G_\infty$, respectively.

\begin{thm}\label{chtopology:thm:open-charts}\label{chtopology:thm:regularity-independence}
The maps $\chi_x$ and $\chi_{x,\infty}$ are homeomorphisms onto open neighborhoods of $[x]$. The inclusion of the smooth solution space, as well as the inclusions between Sobolev orders $k'\geq k\geq4$, induces homeomorphisms of the corresponding unitary gauge quotients.
\end{thm}

\begin{proof}
Near the chosen solution to the Hitchin equations, the Kuranishi parameter map $c_x$ of Proposition~\ref{chlocal:prop:kuranishi-reduction} assigns continuous parameters $(b(y),v(y))$ to nearby Higgs pairs. We choose an open neighborhood $N\subset F^{-1}(0)$ so that $v(y)\in\widetilde U_x\cap U_1$ for $y\in N$, and set
\begin{equation*}
N_{\mathrm{sol}}=N\cap\operatorname{Sol}_{B,k}.
\end{equation*}
The unitary quotient projection restricts to an open surjection $N_{\mathrm{sol}}\to O$, where $O$ is an open neighborhood of $[x]$.

If two elements of $N_{\mathrm{sol}}$ differ by a unitary gauge transformation, their gauge-fixed Higgs pairs differ by a complex gauge transformation. Proposition~\ref{chlocal:prop:local-gauge-uniqueness} therefore puts their parameters in the same $H_x$ orbit. Hence
\begin{equation*}
y\longmapsto\bigl(b(y),q_x(v(y))\bigr)
\end{equation*}
descends to a continuous map $\sigma_x:O\to B\times\mathcal W_x$. Proposition~\ref{chtopology:prop:graded-classes} shows that $y$ and the solution to the Hitchin equations over $R(b(y),v(y))$ determine the same unitary gauge equivalence class. Thus
\begin{equation*}
\chi_x\circ\sigma_x=\id_O.
\end{equation*}
On the open set $\mathcal U_x=\chi_x^{-1}(O)$, injectivity gives $\sigma_x\circ\chi_x=\id_{\mathcal U_x}$. Hence $\chi_x|_{\mathcal U_x}$ is a homeomorphism onto $O$. We choose an open product $B'\times\mathcal W'_x\subset\mathcal U_x$, with $\mathcal W'_x$ Stein. Replacing $B\times\mathcal W_x$ by this product gives an open chart at order $k$.

Lemma~\ref{chtopology:lem:sobolev-regularity} and the smoothness of gauge transformations between smooth solutions show that the inclusion of smooth solutions induces a continuous bijection on the quotient spaces. On each chart its inverse is $\chi_{x,\infty}\circ\chi_x^{-1}$, which is continuous. These charts cover the quotient, so this inclusion is a homeomorphism and the smooth charts are open. Inclusions between finite Sobolev orders commute with these homeomorphisms and hence are homeomorphisms as well.
\end{proof}

We now compare the unitary and complex gauge quotient topologies.

\begin{prop}
\label{chtopology:prop:harmonic-correspondence}
For every $k\geq4$, each polystable Higgs pair of class $\W{k}$ is equivalent under the complex gauge group to a solution to the Hitchin equations, whose unitary gauge equivalence class is unique. The assignment
\begin{equation*}
[(\nabla,\varphi)]_{\mathcal G_{k+1}}
\longmapsto
[(\bar\partial_\nabla,\varphi)]_{\mathcal G_{k+1}^{\C}}
\end{equation*}
defines a homeomorphism
\begin{equation*}
\operatorname{Sol}_{B,k}/\mathcal G_{k+1}
\xrightarrow{\ \sim\ }
\cP_{B,k}^{\mathrm{ps}}/\mathcal G_{k+1}^{\C}.
\end{equation*}
The corresponding map
\begin{equation*}
\operatorname{Sol}_{B,\infty}/\mathcal G_\infty
\xrightarrow{\ \sim\ }
\cP_{B,\infty}^{\mathrm{ps}}/\mathcal G_\infty^{\C}
\end{equation*}
is also a homeomorphism. These homeomorphisms are compatible with the natural inclusions between Sobolev classes and with the inclusion of smooth pairs into $\W{k}$ pairs.
\end{prop}

\begin{proof}
By the Hitchin--Kobayashi correspondence, every polystable complex gauge orbit contains a solution to the Hitchin equations, unique up to the action of the unitary gauge group. Lemma~\ref{chtopology:lem:sobolev-regularity} supplies representatives and gauge transformations of the required regularity. The maps in the statement are therefore continuous bijections.

It remains to prove continuity of the inverse. Let $y_0$ be a polystable representative. Lemma~\ref{chtopology:lem:sobolev-regularity} and the Hitchin--Kobayashi correspondence give a fixed complex gauge transformation $g$ of determinant one for which $y_0^g=x$ is a smooth solution to the Hitchin equations. We choose a neighborhood $N$ of $x$ on which the Kuranishi parameter map $c_x$ takes values in $B\times(\widetilde U_x\cap U_1)$. The map $y\mapsto y^g$ is a homeomorphism of the Sobolev space of relative Higgs pairs. On the open neighborhood $\{y:y^g\in N\}$ of $y_0$, write $(b(y^g),v(y^g))=c_x(y^g)$. The assignment
\begin{equation*}
 y\longmapsto\chi_x\bigl(b(y^g),q_x(v(y^g))\bigr)
\end{equation*}
is continuous. On polystable pairs, this assignment is the inverse correspondence by Proposition~\ref{chtopology:prop:graded-classes}. Such neighborhoods cover $\cP_{B,k}^{\mathrm{ps}}$, so the inverse correspondence composed with the complex gauge quotient projection is continuous. The defining property of the quotient topology proves continuity of the inverse. For smooth pairs, the gauge transformation $g$ may be chosen smooth, and the same argument uses $\chi_{x,\infty}$ and the $C^\infty$ topology.

The inclusion maps induce homeomorphisms of the unitary gauge quotients by Theorem~\ref{chtopology:thm:regularity-independence}. The preceding homeomorphisms commute with the natural inclusions, proving compatibility on the complex gauge quotients.
\end{proof}

Via Proposition~\ref{chtopology:prop:harmonic-correspondence}, the local maps $\chi_x$ may be regarded as maps to the underlying set of $\cM$.

\begin{cor}
\label{chtopology:cor:reference-data}
For each $x\in\cM$, the map $\chi_x$ is a homeomorphism onto an open neighborhood of $x$. These maps determine a Hausdorff topology on $\cM$ for which the projection
\begin{equation*}
p:\cM\longrightarrow\cT_g
\end{equation*}
is continuous. This topology is independent of the Sobolev index, the smooth identifications of the marked fibers with $\Sigma$, the identifications of their underlying bundles preserving the determinant trivializations, and the reference Hermitian metrics.
\end{cor}

\begin{proof}
A complex vector bundle of fixed rank on $\Sigma$ is classified by its first Chern class~\cite[\S6, following (6.22)]{AB83}. The prescribed trivialization of $\det E$ gives $c_1(E)=0$, so $E$ is smoothly trivial. If a smooth global frame has determinant $f$ times the prescribed section, multiplying its first element by $f^{-1}$ produces a frame preserving the determinant trivialization. Two such identifications differ by a smooth complex gauge transformation of determinant one and therefore determine the same point of the complex gauge quotient.

Changing the reference Hermitian metric leaves the underlying Dolbeault operator and Higgs field unchanged. On the compact surface $\Sigma$, the corresponding Sobolev norms are equivalent. By Proposition~\ref{chtopology:prop:harmonic-correspondence}, the two unitary quotient topologies both agree with the topology of the complex gauge quotient. The same proposition and Theorem~\ref{chtopology:thm:regularity-independence} give independence of the Sobolev index.

On the overlap of two Teichm\"uller charts, Lemma~\ref{chlocal:lem:marked-base} gives a marked holomorphic isomorphism of the curve families whose underlying diffeomorphisms depend smoothly on the base parameter. Pullback of Dolbeault operators and Higgs fields is jointly continuous in every finite Sobolev order. Uniform local operator bounds establish this continuity for smooth pairs, and approximation by smooth sections gives it at finite Sobolev regularity. Applying the same argument to the inverse family shows that the induced map between the quotient spaces is a homeomorphism. The corresponding conclusion in the smooth topology follows by applying these estimates to every derivative. Thus the local quotient topologies agree on overlaps, and the projection to $\cT_g$ is continuous.

Over a fixed Teichm\"uller chart, the quotient of the solution space of the Hitchin equations by the unitary gauge group of determinant one is Hausdorff by Proposition~\ref{chlocal:prop:unitary-properness}. Proposition~\ref{chtopology:prop:harmonic-correspondence} transfers this property to the polystable complex gauge quotient. Points lying over distinct elements of $\cT_g$ are separated by disjoint Teichm\"uller neighborhoods. Hence the resulting topology on $\cM$ is Hausdorff.
\end{proof}
\section{Gluing local models and holomorphic families}\label{sec:transitions}

The quotient charts from Section~\ref{sec:tech-chtopology} will be glued by holomorphic transition maps. We use local completeness of the Kuranishi families to construct these transitions, and analytic Hilbert quotients to descend them.

At each point $x$, we choose the domains in Proposition~\ref{chlocal:prop:kuranishi-reduction} and Theorem~\ref{chtopology:thm:open-charts} compatibly, with $\widetilde U_x\subset H_x(Q_x\cap U_1)$. We retain $q_B=\id_B\times q_x$ and the extended family $\widetilde\theta_x$; the map $\chi_x\colon B\times\mathcal W_x\to\cM$ is a homeomorphism onto an open subset. To distinguish constructions at different centers, we write $A_x$ and $u_x$ for the Kuranishi and gauge-fixing maps. The parameter map $c_x$ is given by~\eqref{chlocal:eq:classifying-map}.

Holomorphic maps from a reduced complex space are determined by their underlying maps of sets and factor through the reductions of their targets.

\subsection{Local completeness and transition maps}

\begin{prop}
\label{prop:nearby-completeness}
Let $z_*=(b_*,v_*)$ lie in the chosen Kuranishi slice at $x$, and let $b_S\colon S\to B\subset\cT_g$ be the classifying map of the underlying marked curve family on a reduced analytic germ $(S,s_*)$. If the fiber over $s_*$ is identified with $\theta_x(z_*)$, then, after replacing $S$ by a neighborhood of $s_*$, there is a holomorphic map
\begin{equation*}
\gamma\colon S\longrightarrow Z_x,\qquad
\gamma(s_*)=z_*,
\qquad
\operatorname{pr}_B\circ\gamma=b_S,
\end{equation*}
and the given family is isomorphic to $\gamma^*\theta_x$.
\end{prop}

\begin{proof}
Lemma~\ref{chlocal:lem:smooth-trivialization} gives coefficients $(b_S(s),\alpha(s))$ whose value at $s_*$ is $(b_*,A_x(z_*))$. After restricting $S$ further, we define
\begin{equation*}
\gamma(s)=c_x(b_S(s),\alpha(s)).
\end{equation*}
Since $A_x(z_*)$ satisfies the Coulomb gauge condition, uniqueness of $u_x$ implies $u_x(b_*,A_x(z_*))=0$ and hence $\gamma(s_*)=z_*$. The holomorphic parameter dependence in Lemma~\ref{chlocal:lem:smooth-trivialization}, followed by the holomorphic map $c_x$, proves that $\gamma$ is holomorphic on $S$. Gauge invariance implies that $\gamma$ takes values in $Z_x$. With the trivialization $\widehat\Xi_s$ of that lemma, the isomorphism from the given family to $\gamma^*\theta_x$ is
\begin{equation*}
\exp(-u_x(b_S(s),\alpha(s)))\circ\widehat\Xi_s.
\end{equation*}
The holomorphicity argument in the proof of Theorem~\ref{chlocal:thm:relative-family} applies to this composite gauge transformation.
\end{proof}

\begin{proof}[Proof of Theorem~\ref{thm:analytic-space}]
We use the gluing argument of~\cite[Theorem~4.1 and Lemmas~4.2--4.4]{Fan2020Construction}, together with Proposition~\ref{prop:nearby-completeness}. For $i=1,2$, write $Z_i=Z_{x_i}$, $\chi_i=\chi_{x_i}$, $\mathcal W_i=\mathcal W_{x_i}$, and $\widetilde U_i=q_{x_i}^{-1}(\mathcal W_i)$. We denote the product quotient map by
\begin{equation*}
q_{B_i}=\id_{B_i}\times q_{x_i}\colon
B_i\times\widetilde U_i\longrightarrow B_i\times\mathcal W_i.
\end{equation*}
Here $q_{x_i}\colon Q_{x_i}\to Y_{x_i}$ is the analytification of the affine GIT quotient map. By the local universal property of the Teichm\"uller family, the marked curve families are holomorphically isomorphic over an overlap of the Teichm\"uller charts. At a point in the overlap of the two moduli charts, choose representatives $z_i\in Z_i\cap(B_i\times\widetilde U_i)$ on the corresponding closed $H_{x_i}$-orbits. Such representatives exist because each orbit in $\widetilde U_i$ meets the small slice. Proposition~\ref{chtopology:prop:graded-classes} shows that the associated Higgs bundles are polystable and isomorphic, and Lemma~\ref{chtopology:lem:sobolev-regularity} shows that every isomorphism represented by a $\W{k+1}$ complex gauge transformation is smooth.

Proposition~\ref{prop:nearby-completeness} gives a holomorphic map $\gamma\colon\mathcal V_1\to Z_2$ on a neighborhood $\mathcal V_1$ of $z_1$, with $\gamma(z_1)=z_2$, commuting with the maps to $\cT_g$ and inducing an isomorphism of the corresponding families. We may assume that $\gamma(\mathcal V_1)\subset B_2\times\widetilde U_2$. If $z,z'\in\mathcal V_1$ belong to the same $H_{x_1}$-orbit, their images determine isomorphic Higgs bundles. Corollary~\ref{chlocal:cor:extended-stabilizers} gives
\begin{equation*}
q_{B_2}(\gamma(z))=q_{B_2}(\gamma(z')).
\end{equation*}
Hence
\begin{equation*}
\widehat\tau_{21}(\sigma\cdot z)=q_{B_2}(\gamma(z)),
\qquad z\in\mathcal V_1,\quad \sigma\in H_{x_1},
\end{equation*}
defines an $H_{x_1}$-invariant map on $H_{x_1}\cdot\mathcal V_1$. Its restrictions to the translates of $\mathcal V_1$ are holomorphic and agree on their overlaps, so $\widehat\tau_{21}$ is holomorphic.

The analytic Hilbert quotient $q_{B_1}$ sends invariant relatively closed subsets to closed subsets. It follows that an invariant open set containing a closed orbit contains the inverse image of a neighborhood in the quotient~\cite[Section~3.1]{HeinznerHuckleberry1999}. Hence there is a neighborhood $O_1\subset B_1\times\mathcal W_1$ of $q_{B_1}(z_1)$ such that
\begin{equation*}
q_{B_1}^{-1}(O_1)\subset H_{x_1}\cdot\mathcal V_1.
\end{equation*}
After shrinking $O_1$ so that $\chi_1(O_1)$ lies in the overlap of the two moduli charts, Theorem~\ref{chlocal:prop:invariant-descent} gives a holomorphic map
\begin{equation*}
\tau_{21}\colon O_1\longrightarrow B_2\times\mathcal W_2,
\qquad
\widehat\tau_{21}=\tau_{21}\circ q_{B_1}.
\end{equation*}

For $t\in O_1$, we choose a representative $z\in\mathcal V_1$ of the closed orbit over $t$. The Higgs bundles represented by $z$ and $\gamma(z)$ are polystable and isomorphic. Proposition~\ref{chtopology:prop:graded-classes} then gives
\begin{equation*}
\chi_2\circ\tau_{21}(t)=\chi_1(t).
\end{equation*}
The construction with the indices reversed produces the inverse transition. These transitions commute with the projections to $\cT_g$. On triple overlaps they satisfy
\begin{equation*}
\tau_{32}\circ\tau_{21}=\tau_{31},
\end{equation*}
because both sides represent the same change of moduli charts.

These transitions glue the local reduced complex-analytic structures on the Hausdorff quotient and preserve the projection to $\cT_g$. The resulting space is second-countable: a countable cover by Teichm\"uller charts suffices, and over each chart the quotient of the solution space of the Hitchin equations by the unitary gauge group of determinant one is the image, under an open quotient map, of a subspace of a separable Sobolev space of pairs $(\nabla,\varphi)$. The fixed smooth reference bundle in Corollary~\ref{chtopology:cor:reference-data} gives this description, and the images of a countable basis form a countable basis of the quotient. Combining two such atlases proves independence of the slices. Lemma~\ref{chlocal:lem:kuranishi-regularity}, Theorem~\ref{chtopology:thm:regularity-independence}, and Proposition~\ref{chtopology:prop:harmonic-correspondence} show that the smooth and Sobolev constructions give the same analytic structure and that the complex and unitary gauge quotient topologies agree.
\end{proof}

\subsection{Normality}\label{sec:normality}

The trace splitting relates our local model to the $\GL_r(\C)$-Higgs moduli space on $X$. We denote the scalar deformation space by
\begin{equation*}
S_X:=H^1(X,\cO_X)\oplus H^0(X,K_X).
\end{equation*}

\begin{prop}\label{prop:scalar-splitting}\label{prop:scalar-quotient}\label{prop:determinant-normality}
The germ of the $\GL_r(\C)$-Higgs moduli space on $X$ at the point obtained from $x$ by forgetting the determinant trivialization and trace-free condition is $(Y_x\times S_X,([0],0))$. The space $Y_x$ is normal.
\end{prop}

\begin{proof}
The trace splitting $\End E=\End_0E\oplus\cO_X\Id_E$ splits the deformation complex and its harmonic operators, with scalar deformation subcomplex $[\cO_X\xrightarrow{0}K_X]$. The graded bracket of a scalar form with any endomorphism-valued form vanishes, and every graded commutator has zero trace. Hence the quadratic equation depends only on the trace-free summand, and its reduced zero locus is $Q_x\times S_X$. Moreover,
\begin{equation*}
\Aut(E,\varphi)=\C^*H_x.
\end{equation*}
Indeed, division by an $r$th root of the constant determinant places every automorphism in $\C^*H_x$. The scalar factor acts trivially on $Q_x$. Consequently
\begin{equation*}
\C[Q_x\times S_X]^{\Aut(E,\varphi)}
=\C[Q_x]^{H_x}\otimes\C[S_X],
\end{equation*}
and, after analytification, the affine quotients satisfy
\begin{equation*}
(Q_x\times S_X)/\!/\Aut(E,\varphi)\simeq Y_x\times S_X.
\end{equation*}
By~\cite[Theorem~5.3 and Lemma~6.7]{Fan2020Construction}, this is the $\GL_r(\C)$-Higgs moduli germ and is normal near $([0],0)$.

We denote the projection by $\operatorname{pr}\colon Y_x\times S_X\to Y_x$ and the zero section by $\iota\colon Y_x\to Y_x\times S_X$. For $y$ near $[0]$, the induced maps on local rings satisfy
\begin{equation*}
\iota^\#\circ\operatorname{pr}^\#=\id_{\cO_{Y_x,y}}.
\end{equation*}
Thus $\operatorname{pr}^\#$ is injective. Since $\cO_{Y_x\times S_X,(y,0)}$ is a normal local domain, $\cO_{Y_x,y}$ is a domain as well. If $a/b\in\operatorname{Frac}(\cO_{Y_x,y})$ is integral over $\cO_{Y_x,y}$, then its image under $\operatorname{pr}^\#$ is integral over the normal local ring $\cO_{Y_x\times S_X,(y,0)}$. Hence
\begin{equation*}
\operatorname{pr}^\#a=\operatorname{pr}^\#b\,c'
\end{equation*}
for some $c'\in\cO_{Y_x\times S_X,(y,0)}$. Applying $\iota^\#$ gives $a=b\,\iota^\#c'$, so $\cO_{Y_x,y}$ is normal. Thus $Y_x$ is normal near $[0]$. The positive grading on $\C[Q_x]^{H_x}$ has degree-zero part $\C$, and the induced $\C^*$-action contracts $Y_x$ to $[0]$. Since every nonzero scalar acts by an automorphism, normality holds throughout $Y_x$.
\end{proof}

\begin{proof}[Proof of Theorem~\ref{thm:normal-model}]
By Theorem~\ref{chlocal:thm:saturated-quotient}, the holomorphic quotient chart gives $(\cM,x)\simeq(B\times Y_x,(0,[0]))$ over $B\subset\cT_g$. These germs are normal by the preceding proposition and preservation of normality under product with a complex manifold. The stabilizer assertion follows from Corollary~\ref{chlocal:cor:extended-stabilizers}.
\end{proof}

\subsection{Semistable families and the coarse moduli property}\label{sec:coarse}

To classify semistable families, we construct local lifts near the Seshadri graded object of the central fiber.

\begin{prop}
\label{prop:semistable-lifting}
Let $(S,s_0)$ carry a family in Definition~\ref{defn:family-functors}, and let $x$ be the Seshadri graded object of its central fiber. After replacing $S$ by a neighborhood of $s_0$, there exist a holomorphic map
\begin{equation*}
\gamma\colon S\longrightarrow B\times(\widetilde U_x\cap U)
\end{equation*}
and a holomorphic isomorphism from the given family to $\gamma^*\widetilde\theta_x$ preserving the determinant trivialization.
\end{prop}

\begin{proof}
Apply Lemma~\ref{chtopology:lem:seshadri-degeneration} to the central fiber. Identify its graded bundle smoothly with the harmonic representative of $x$, preserving the induced determinant trivialization, and transport the transformations $g_t$ of~\eqref{chtopology:eq:degeneration-weights} to this fixed bundle. The transformed Dolbeault operators and Higgs fields then converge smoothly to the representative of $x$.

Lemma~\ref{chlocal:lem:smooth-trivialization} extends the chosen identification on the central fiber to a smooth trivialization $\widehat\Xi_s$ of the family, giving coefficients $(b(s),\alpha(s))$ with $b(s_0)=0$. Continuity of $c_x$ provides a neighborhood $\mathcal V$ of $(0,0)$ in $F^{-1}(0)$ such that
\begin{equation*}
c_x(\mathcal V)
\subset B\times(\widetilde U_x\cap U).
\end{equation*}
We choose $t>0$ so that $(0,\alpha(s_0)^{g_t})\in\mathcal V$ and apply this fixed gauge transformation to the entire family.

The action of this fixed gauge transformation is holomorphic by~\eqref{chlocal:eq:right-gauge}. Its term involving $\mu(b)$ is $-g_t^{-1}\iota_{\mu(b)}\partial_hg_t$ and vanishes at $s_0$ because $\mu(b(s_0))=0$. Continuity in a fixed $\W{k}$ norm, with $k\geq4$, gives
\begin{equation*}
\bigl(b(s),\alpha(s)^{g_t}\bigr)\in\mathcal V
\end{equation*}
after restricting $S$. We then set
\begin{equation*}
\gamma(s)=c_x\bigl(b(s),\alpha(s)^{g_t}\bigr).
\end{equation*}
Gauge covariance and the choice of $\mathcal V$ imply that $\gamma$ takes values in $B\times(\widetilde U_x\cap U)$. The right gauge action gives the following isomorphism from the given family to $\gamma^*\widetilde\theta_x$:
\begin{equation*}
\exp\!\left(-u_x\bigl(b(s),\alpha(s)^{g_t}\bigr)\right)g_t^{-1}\circ\widehat\Xi_s.
\end{equation*}
It has determinant one and is holomorphic by the argument in Proposition~\ref{prop:nearby-completeness}. Proposition~\ref{chtopology:prop:graded-classes} gives $q_B(\gamma(s_0))=(0,[0])$. If the central fiber is nonpolystable, then $\gamma(s_0)\ne(0,0)$, since the fiber at $(0,0)$ is polystable.
\end{proof}

\begin{proof}[Proof of Theorem~\ref{thm:coarse}]
For a local lift $\gamma$ supplied by Proposition~\ref{prop:semistable-lifting}, Proposition~\ref{chtopology:prop:graded-classes} shows that
\begin{equation*}
\chi_x\circ q_B\circ\gamma
\end{equation*}
classifies the Seshadri graded objects of the fibers. These local maps agree on overlaps and define $\eta$. The pointwise description proves naturality, compatibility with the map to $\cT_g$, and equality precisely for fiberwise S-equivalent families. Polystable families over a point give $\FS(\pt)\simeq\cM$.

We fix a natural transformation $\Theta\colon\Fss\to\Hol(-,T)$ and apply it to $\mathcal E_x=\widetilde\theta_x$ on $B\times\widetilde U_x$, whose fibers are semistable. For every $\sigma\in H_x$, pullback by $(b,v)\mapsto(b,v^\sigma)$ gives a family isomorphic to $\mathcal E_x$. Hence
\begin{equation*}
\Theta_x:=\Theta_{B\times\widetilde U_x}(\mathcal E_x)
\end{equation*}
is $H_x$-invariant. Theorem~\ref{chlocal:prop:invariant-descent} gives
\begin{equation*}
\Theta_x=\overline\Theta_x\circ q_B
\end{equation*}
for a holomorphic map $\overline\Theta_x\colon B\times\mathcal W_x\to T$, as in~\cite[Proposition~3.9 and Corollary~3.10]{Fan2020QuasiProjectivity}. Thus $\Theta_x$ is constant on every fiber of $q_B$.

For $y\in B\times\mathcal W_x$, we choose $z\in q_B^{-1}(y)$ in the unique closed $H_x$-orbit. Naturality gives
\begin{equation*}
\overline\Theta_x(y)=\Theta_{\pt}([(\mathcal E_x)_z]),
\end{equation*}
which depends only on the isomorphism class of the marked polystable Higgs bundle. The maps $\overline\Theta_x$ thus agree on chart overlaps and define $f\colon\cM\to T$.

For a family locally isomorphic to $\gamma^*\mathcal E_x$, naturality gives
\begin{equation*}
\Theta_S(\mathcal E)
=\Theta_x\circ\gamma
=\overline\Theta_x\circ q_B\circ\gamma
=f\circ\eta_S(\mathcal E).
\end{equation*}
These local identities prove the required factorization and show that every such transformation takes equal values on fiberwise S-equivalent families. Polystable families over a point determine $f$ uniquely. Since $\eta$ factors through $\FS$, pulling back transformations along $\Fss\to\FS$ proves the corresponding universal property for $\FS$.
\end{proof}

\section{Natural actions and comparison on the stable locus}\label{sec:actions}

The coarse moduli property determines the natural actions, while the local models describe the stable locus.

\begin{cor}\label{prop:natural-actions}\label{prop:stable-charts}
The mapping class group $\Gamma_g$ and $\C^*$ act holomorphically on $\cM$ by changes of marking and Higgs-field scaling, respectively. These actions commute. The stable locus is open, and at each stable point $x$,
\begin{equation*}
H_x=\{\zeta\Id_E:\zeta^r=1\},\qquad Q_x=V_x.
\end{equation*}
The local Kuranishi charts on the stable locus are complex manifolds of the form $B\times U$, and the associated families are holomorphic.
\end{cor}

\begin{proof}
We fix $[\rho]\in\Gamma_g$ and choose a representative $\rho\in\Diff^+(\Sigma)$. For a marking $\mathfrak m\colon\Sigma\to X$, the change
\begin{equation*}
\mathfrak m\longmapsto\mathfrak m\circ\rho^{-1}
\end{equation*}
acts on holomorphic marked families. Applying the coarse moduli property to $\widetilde\theta_x$ gives an $H_x$-invariant classifying map on $B\times\widetilde U_x$, which descends through $q_B$. These local maps agree on overlaps, and the construction with $\rho^{-1}$ gives the inverse.

For the scaling action, we use the family $(E,\lambda\varphi)$ on $\C^*\times B\times\widetilde U_x$. Scaling preserves invariant subbundles, stability, and graded objects. By the product argument in Theorem~\ref{chlocal:thm:saturated-quotient}, applied locally in $\C^*$, the map $\id_{\C^*}\times q_B$ is an analytic Hilbert quotient. The invariant classifying map therefore descends holomorphically to $\C^*\times B\times\mathcal W_x$, and these maps glue over the moduli charts. The group laws, commutation relation, and compatibility with the actions on $\cT_g$ follow from their pointwise descriptions. Scaling preserves the automorphism group, and a change of marking induces an isomorphism of automorphism groups.

Since a stable Higgs bundle is simple, $H_x=\{\zeta\Id_E:\zeta^r=1\}$ acts trivially on the deformation complex and has zero Lie algebra. Serre duality implies $H^2=0$ and therefore $Q_x=V_x$. Every nearby orbit is closed, so Proposition~\ref{chtopology:prop:graded-classes} identifies its Higgs bundle as polystable; Corollary~\ref{chlocal:cor:local-stabilizers} identifies its stabilizer with the same scalar group. A strictly polystable bundle would have two stable summands of ranks $r_1,r_2$. Multiplication by $t^{r_2}$ and $t^{-r_1}$ on these summands, with the identity on the others, would define a positive-dimensional group of automorphisms of determinant one. This contradicts the stabilizer calculation, so the chart consists of stable Higgs bundles.
\end{proof}

\begin{cor}\label{cor:flat-stable}
The map $p\colon\cM\to\cT_g$ is flat with normal analytic fibers. Its restriction to $\cMs$ is a holomorphic submersion of relative dimension $2(r^2-1)(g-1)$. At each stable point $x=(X,E,\varphi)$, there is a canonical exact sequence
\begin{equation}\label{eq:stable-tangent-sequence}
0\longrightarrow V_x\longrightarrow T_x\cMs
\xrightarrow{\,dp_x\,}H^1(X,T_X)\longrightarrow0.
\end{equation}
In particular, $\dim_{\C}\cMs=(2r^2+1)(g-1)$.
\end{cor}

\begin{proof}
Theorem~\ref{thm:normal-model} gives the germ of $p$ as the product projection $B\times Y_x\to B$. Such projections are flat~\cite[Corollary~I.1.88]{GreuelLossenShustin2007}, and their fiber germs are normal because $Y_x$ is normal.

At a stable point, Corollary~\ref{prop:stable-charts} gives $Y_x=V_x$ and a product neighborhood contained in $\cMs$. Thus $p|_{\cMs}$ is a holomorphic submersion. Its vertical tangent space is $V_x$ by the deformation complex in Definition~\ref{chlocal:defn:deformation-complex}. The Kodaira--Spencer isomorphism in Lemma~\ref{chlocal:lem:marked-base} gives $T_X\cT_g\simeq H^1(X,T_X)$, proving~\eqref{eq:stable-tangent-sequence}.

The degree-zero and degree-two hypercohomology groups vanish at a stable point, as in the proof of Corollary~\ref{prop:stable-charts}. Riemann--Roch therefore gives
\begin{equation*}
\dim V_x=-\chi(C_x^\bullet)
=\chi(\End_0E\otimes K_X)-\chi(\End_0E)
=2(r^2-1)(g-1).
\end{equation*}
Adding $\dim\cT_g=3g-3$ gives the stated dimension of $\cMs$.
\end{proof}

For $r\geq2$, let $P_{\mathrm{SL}}$ be the principal $\SL_r(\C)$ frame bundle defined by the chosen determinant trivialization. Stability for principal and associated vector Higgs bundles agrees~\cite[Section~4.2, Remark~4.3 and Theorem~4.4]{GarciaPradaGothenMundet2012}. A marked holomorphic curve automorphism isotopic to the identity is trivial. Together with Corollary~\ref{prop:stable-charts}, this places our stable objects in the regularly stable locus of~\cite{CTW25}.

In the notation of~\cite{CTW25}, $J$ is the complex structure on $P_{\mathrm{SL}}$ compatible with the principal action, and $\Phi$ is the relative Higgs field. The complex structure on their configuration space is
\begin{equation*}
I_{(J,\Phi)}(\delta J,\delta\Phi)=(J\delta J,i\delta\Phi)
\end{equation*}
by~\cite[Section~3.1, (3.2) and the discussion following Lemma~3.3]{CTW25}. Under $\Phi=2i\varphi$, one has $[\Phi,\Phi^{*h}]=4[\varphi,\varphi^{*h}]$, so their equation
\begin{equation}\label{eq:ctw-normalization}
F_\nabla+\frac14[\Phi,\Phi^{*h}]=0
\end{equation}
agrees with the real Hitchin equation used here. The holomorphic families defining the two local atlases give classifying maps in both directions.

\begin{prop}
\label{prop:ctw-family-comparison}
The slices constructed in~\cite[Proposition~4.19]{CTW25} define holomorphic marked Higgs families. Conversely, the families over our Kuranishi charts on the stable locus define holomorphic maps to the stable joint moduli space constructed in~\cite{CTW25}. Under $\Phi=2i\varphi$, both maps classify the same marked Higgs bundles.
\end{prop}

\begin{proof}
We choose a parameter polydisc $S$ in one of the slices from~\cite[Proposition~4.19]{CTW25}. The map $s\mapsto(J_s,\Phi_s)$ is holomorphic into their Sobolev configuration manifold, and its values are smooth on $P_{\mathrm{SL}}$ by~\cite[Lemma~4.20]{CTW25}. Thus, for every real tangent vector $\dot s$ to $S$,
\begin{equation*}
d_{i\dot s}J=Jd_{\dot s}J,\qquad
d_{i\dot s}\Phi=i\,d_{\dot s}\Phi.
\end{equation*}
Each $J_s$ is integrable and induces the complex structure of the curve through the bundle projection. We define an almost complex structure on $P_{\mathrm{SL}}\times S$ by
\begin{equation*}
\mathcal J(u,\dot s)=(J_su,i\dot s).
\end{equation*}
Its Nijenhuis tensor vanishes on pairs tangent to $P_{\mathrm{SL}}$ and on pairs tangent to $S$. For $u\in TP_{\mathrm{SL}}$ and $\dot s\in TS$, the mixed component is
\begin{equation*}
-(d_{i\dot s}J)u+J_s(d_{\dot s}J)u=0.
\end{equation*}
The Newlander--Nirenberg theorem therefore gives an integrable complex structure on $P_{\mathrm{SL}}\times S$~\cite{NewlanderNirenberg1957}. The induced complex structure on $\Sigma\times S$ is integrable by the same mixed calculation, and $\Sigma\times S\to S$ is a proper holomorphic curve family. The projection from $P_{\mathrm{SL}}\times S$ and the principal $\SL_r(\C)$ action are holomorphic, so the associated vector bundle is holomorphic and carries its induced holomorphic determinant trivialization. The form $\Phi_s$ has relative type $(1,0)$, and its relative Dolbeault derivative vanishes by the equations defining the slice. The fiberwise equation together with $d_{i\dot s}\Phi=i\,d_{\dot s}\Phi$ gives the total $\bar\partial$-equation for $\Phi$. Hence $\Phi/(2i)$ is a holomorphic relative Higgs field, whose classifying map to $\cM$ is supplied by Theorem~\ref{thm:coarse}.

Conversely, let $S$ be a parameter polydisc in one of our stable Kuranishi charts. The holomorphic family constructed in Theorem~\ref{chlocal:thm:relative-family} induces a holomorphic principal frame bundle. In its chosen smooth trivialization, the total $(0,1)$-operator has the standard parameter directions $\partial_{\bar s_j}$. Thus its integrable complex structure on $P_{\mathrm{SL}}\times S$ has the form $\mathcal J(u,\dot s)=(J_su,i\dot s)$. The mixed Nijenhuis equation on $TP_{\mathrm{SL}}\times TS$ gives
\begin{equation*}
-(d_{i\dot s}J)u+J_s(d_{\dot s}J)u=0
\qquad (u\in TP_{\mathrm{SL}}),
\end{equation*}
and hence $d_{i\dot s}J=J_s d_{\dot s}J$.
Writing $b(s)$ for the curve parameter, we express the Higgs field in a fixed smooth bundle frame over a coordinate $z$ on the central fiber as
\begin{equation*}
\Phi_s=2i\varphi_{\mathrm{tot},z}(s)\bigl(dz+\mu(b(s))\,d\bar z\bigr).
\end{equation*}
Both $\varphi_{\mathrm{tot},z}(s)$ and $\mu(b(s))$ depend holomorphically on $s$. Hence
\begin{equation*}
d_{i\dot s}\Phi=i\,d_{\dot s}\Phi.
\end{equation*}
These Cauchy--Riemann identities show that $s\mapsto(J_s,\Phi_s)$ is a holomorphic map into the configuration space of~\cite{CTW25}. The quotient map on the regularly stable locus is holomorphic by~\cite[Section~4.8]{CTW25}, so composition gives the comparison map on $S$. These maps agree on overlaps and define the comparison on the stable locus. For a family over an arbitrary reduced analytic base, local completeness and compatibility with base change give the corresponding holomorphic map after restricting the base.
\end{proof}

\begin{proof}[Proof of Theorem~\ref{thm:actions-comparison}]
Corollary~\ref{prop:natural-actions} establishes the actions and describes the stable locus. The two holomorphic comparison maps induced by the families in Proposition~\ref{prop:ctw-family-comparison} preserve the marked Higgs bundles. Their compositions are the identity, yielding the required biholomorphism over $\cT_g$. This biholomorphism is equivariant for changes of marking and Higgs scaling.
\end{proof}
\section{Relative algebraic comparison and the Hitchin map}\label{sec:algebraic-comparison}

We compare the analytic and algebraic moduli spaces by classifying holomorphic families. Normality of the algebraic parameter schemes allows us to apply Theorem~\ref{thm:coarse} to their universal families after pullback to Teichm\"uller neighborhoods. The resulting comparison then transfers properness of the relative algebraic Hitchin morphism to $\cM$.

\subsection{Comparison with the algebraic construction}

For an integer $m\geq3$, let $\pi_m\colon\mathcal X_m\to S_m=\mathcal M_g[m]$ be the universal curve with level-$m$ structure and let $a_m\colon\cT_g\to S_m^{\mathrm{an}}$ be the covering map induced by markings. The scheme $S_m$ is smooth and quasi-projective, and $\pi_m$ is smooth and projective~\cite[Section~7.3 and Propositions~7.10--7.12]{AlessandriniCollierMaximalComponents}. We set
\begin{equation}\label{eq:relative-algebraic-space}
D_m=M_{\mathrm{Dol}}(\mathcal X_m/S_m,\SL_r),\qquad
N_m=\cT_g\times_{S_m^{\mathrm{an}}}D_m^{\mathrm{an}}.
\end{equation}
We use the scheme structure constructed in~\cite[Proposition~9.7]{SimpsonModuliPartII}. The map from $D_m$ to the relative $\GL_r$ moduli scheme is finite by~\cite[Corollary~9.19]{SimpsonModuliPartII}; hence $D_m$ is a separated quasi-projective scheme. The space $N_m$ is the marked analytic moduli space of~\cite[Theorem~7.5 and Section~7.4]{AlessandriniCollierMaximalComponents}.

\begin{lem}\label{lem:relative-framed-normality}
Let $S\to S_m$ be an \'etale neighborhood on which $\pi_S\colon\mathcal X_S=\mathcal X_m\times_{S_m}S\to S$ admits a section $\xi$. The framed schemes
\begin{equation*}
R_{\mathrm{GL}}=R_{\mathrm{Dol}}(\mathcal X_S/S,\xi,\GL_r),\qquad
R_{\mathrm{SL}}=R_{\mathrm{Dol}}(\mathcal X_S/S,\xi,\SL_r)
\end{equation*}
are normal.
\end{lem}

\begin{proof}
These schemes represent semistable Higgs bundles with a frame along $\xi$; their construction and good quotients are given in~\cite[Theorem~9.6 and Proposition~9.7]{SimpsonModuliPartII}. Closed orbits for the change of frame correspond to polystable bundles~\cite[Theorem~9.8]{SimpsonModuliPartII}. We first prove normality for $R_{\mathrm{GL}}$, then pass to $R_{\mathrm{SL}}$ by tensoring with line bundles.

Let $z\in R_{\mathrm{GL}}$ lie on a closed orbit over $s\in S$, represented by $(F,\varphi_F,\beta)$ on $X=(\mathcal X_S)_s$. Denote by $\mathcal R_{\mathrm{Higgs}}$ the scheme-theoretic completed local ring at the origin defined by the components of $[v,v]=0$ in
\begin{equation*}
\mathbb H^1\bigl([\End F\xrightarrow{[\varphi_F,\cdot]}\End F\otimes K_X]\bigr).
\end{equation*}
We retain the ideal generated by these equations, without taking its radical. By~\cite[Theorem~10.4]{SimpsonModuliPartII},
\begin{equation}\label{eq:framed-completion}
\widehat{\cO}_{(R_{\mathrm{GL}})_s,z}
\simeq\mathcal R_{\mathrm{Higgs}}[[u_1,\ldots,u_a]]
\end{equation}
for some $a\geq0$. The fiber $(R_{\mathrm{GL}})_s$ is normal for $g\geq2$ by~\cite[Corollary~11.7]{SimpsonModuliPartII}. Excellence of its local rings implies normality of the completion in~\eqref{eq:framed-completion}. Normality of $\mathcal R_{\mathrm{Higgs}}$ follows by faithfully flat descent.

Let $\mathcal R_{\mathrm{pair}}$ be the completed local ring of the Kuranishi space of $(X,F,\varphi_F)$, again retaining its Maurer--Cartan equation ideal before reduction. Section~2.2 of~\cite{Ono2023Structure} identifies this ring with the ring pro-representing the Kuranishi functor on local Artin algebras. The DGLA comparisons in Lemma~4.4 and Corollary~5.2, and the product theorem in Theorem~5.3 of that paper, therefore yield an isomorphism of these unreduced completed local rings:
\begin{equation}\label{eq:joint-completed-ring}
\mathcal R_{\mathrm{pair}}
\simeq\mathcal R_{\mathrm{Higgs}}[[b_1,\ldots,b_{3g-3}]].
\end{equation}
The curve deformation space is smooth of dimension $3g-3$. Since the level structure extends uniquely across nilpotent thickenings, the formal curve family over $S$ is universal, and its chosen section pulls back with the family. The normality of $\mathcal R_{\mathrm{Higgs}}$ and~\eqref{eq:joint-completed-ring} imply that $\mathcal R_{\mathrm{pair}}$ is normal.

The formal Kuranishi family admits a frame along $\xi$, since its restriction to the section is free over the local ring $\mathcal R_{\mathrm{pair}}$. Varying that frame gives a morphism
\begin{equation}\label{eq:formal-framed-family}
\Spf\mathcal R_{\mathrm{pair}}\times\widehat{\GL_r}_1
\longrightarrow\widehat{R_{\mathrm{GL}}}_z.
\end{equation}
We verify formal smoothness with the specified isomorphisms retained. Let $\mathsf A'\twoheadrightarrow\mathsf A$ be a small extension of local Artin $\C$-algebras, with kernel $I$ and $\mathfrak m_{\mathsf A'}I=0$. Suppose a framed deformation over $\mathsf A'$ and a lift of its reduction through~\eqref{eq:formal-framed-family} are given. In the differential graded Lie algebra $\mathfrak d^\bullet$ controlling the joint deformation, the prescribed isomorphism over $\mathsf A$, with identity reduction on the object on the central fiber, is represented by an element of $\exp(\mathfrak d^0\otimes\mathfrak m_{\mathsf A})$. Its logarithm lifts to $\mathsf A'$. Acting by the lifted element arranges that the Maurer--Cartan element reduces to the prescribed Kuranishi element.

Choose the Hodge complement
\begin{equation*}
\mathfrak d^1=B^1(\mathfrak d)\oplus C^1(\mathfrak d),
\qquad B^1(\mathfrak d)=\im d,\qquad C^1(\mathfrak d)=\ker d^*.
\end{equation*}
Let $\eta'$ be the Maurer--Cartan element after the preceding adjustment. Its $B^1(\mathfrak d)$ component, denoted by $\delta$, lies in $B^1(\mathfrak d)\otimes I$, because the prescribed reduction lies in the Kuranishi slice. Choose $u\in\mathfrak d^0\otimes I$ with $du=-\delta$. With the right gauge convention, the action reduces to
\begin{equation*}
(\eta')^{\exp(u)}=\eta'+du=\eta'-\delta,
\end{equation*}
since $\mathfrak m_{\mathsf A'}I=0$ annihilates all nonlinear terms involving $u$. Gauge transformations preserve the Maurer--Cartan equation, and $\exp(u)$ reduces to the identity modulo $I$. The adjusted element therefore defines a Kuranishi parameter extending the given one, with the specified isomorphism over $\mathsf A$ unchanged. Universality of the level curve and $H^0(X,T_X)=0$ provide the compatible curve isomorphism. Comparing the transported frame with the target frame supplies an element of $\GL_r(\mathsf A')$ extending the prescribed frame parameter. This proves formal smoothness of~\eqref{eq:formal-framed-family}.

GAGA over Artin bases~\cite[proof of Proposition~5.4, pp.~113--114]{SimpsonModuliPartI} and analytic representability~\cite[Lemma~5.7]{SimpsonModuliPartI} give the algebraic completion of $R_{\mathrm{GL}}$ as the formal deformation space in~\eqref{eq:formal-framed-family}. Hence
\begin{equation*}
\widehat{\cO}_{R_{\mathrm{GL}},z}
\longrightarrow\mathcal R_{\mathrm{pair}}[[t_1,\ldots,t_{r^2}]]
\end{equation*}
is a local homomorphism formally smooth for the maximal-ideal-adic topologies, hence faithfully flat. Normality descends from its target through this homomorphism and then through completion~\cite[Tags~07NP and~033G]{StacksProject}, proving that $R_{\mathrm{GL}}$ is normal at $z$. Since $R_{\mathrm{GL}}$ is of finite type over $\C$, its nonnormal locus is closed and invariant under changes of frame. In a saturated affine quotient chart, the closure of every orbit meets a closed orbit. The orbit closure of any nonnormal point would contain a nonnormal point on a closed orbit, contradicting normality there. Thus $R_{\mathrm{GL}}$ is normal.

For the passage to $\SL_r$, consider
\begin{equation*}
R_1=R_{\mathrm{Dol}}(\mathcal X_S/S,\xi,\mathbb G_m)
\simeq\Pic^0(\mathcal X_S/S)\times_S\Tot((\pi_S)_*K_{\mathcal X_S/S}).
\end{equation*}
The frame rigidifies the relative Picard functor. The scheme $R_1$ is smooth and surjective over $S$, and the morphism
\begin{equation*}
[r]\colon R_1\longrightarrow R_1,\qquad
(\mathcal L,\psi)\longmapsto(\mathcal L^{\otimes r},r\psi)
\end{equation*}
is finite \'etale: multiplication by $r$ has this property on the relative Picard scheme and is invertible on the vector bundle. Determinant, trace, and the determinant frame define $\delta\colon R_{\mathrm{GL}}\to R_1$. There is an isomorphism of schemes
\begin{equation}\label{eq:rank-one-twisting}
R_{\mathrm{SL}}\times_S R_1
\simeq R_{\mathrm{GL}}\times_{R_1,[r]}R_1.
\end{equation}
Over any $S$-scheme, its forward construction sends
\begin{equation*}
(E,\varphi;\mathcal L,\psi)
\longmapsto(E\otimes\mathcal L,\varphi\otimes1+1\otimes\psi;\mathcal L,\psi)
\end{equation*}
with the tensor product frame. Conversely, a point $(F,\varphi_F;\mathcal L,\psi)$ of the right-hand side gives a unique frame-preserving isomorphism
\begin{equation*}
(\det F,\tr\varphi_F)\simeq(\mathcal L^{\otimes r},r\psi).
\end{equation*}
Uniqueness follows from $(\pi_S)_*\cO_{\mathcal X_S}=\cO_S$, universally, and preservation of the frame. The bundle $E=F\otimes\mathcal L^{-1}$ then has a specified determinant trivialization, the Higgs field $\varphi_F-\psi\Id$ is trace-free, and the induced frame preserves the trivialization. Tensoring by a degree-zero line bundle preserves semistability. These constructions are inverse over every base and prove~\eqref{eq:rank-one-twisting}. Its right-hand side is finite \'etale over $R_{\mathrm{GL}}$, hence normal; its left-hand side maps smoothly and surjectively to $R_{\mathrm{SL}}$. Faithfully flat descent proves normality of $R_{\mathrm{SL}}$.
\end{proof}

\begin{thm}\label{thm:relative-algebraic-comparison}
The scheme $D_m$ and the analytic space $N_m$ are normal. There is a unique biholomorphism over $\cT_g$
\begin{equation}\label{eq:relative-comparison}
I_m\colon\cM\xrightarrow{\sim}N_m
\end{equation}
which sends each marked polystable Higgs bundle to its algebraic moduli class. For every family in $\Fss$, its classifying maps correspond under $I_m$. These isomorphisms are compatible with changes of finite level, changes of marking, and Higgs scaling.
\end{thm}

\begin{proof}
On a chart $S$ as in Lemma~\ref{lem:relative-framed-normality}, the base change of $D_m$ is the good quotient $R_{\mathrm{SL}}/\!/\SL_r$. Invariant rings of normal rings are normal, so the lemma and \'etale descent prove normality of $D_m$. Analytification preserves normality, and $N_m\to D_m^{\mathrm{an}}$ is locally biholomorphic because $a_m$ is. Thus~$N_m$~is~\mbox{normal}.

By~\cite[Lemma~5.7]{SimpsonModuliPartI}, $R_{\mathrm{GL}}^{\mathrm{an}}$ represents holomorphic framed Higgs families. The condition of local freeness in that lemma holds because the stable factors of a Seshadri filtration of a semistable Higgs bundle on a smooth curve are vector bundles. The frame normalization in the proof of Lemma~\ref{lem:relative-framed-normality} realizes $R_{\mathrm{SL}}$ as the fiber of $R_{\mathrm{GL}}\to R_1$ over the trivial framed rank-one Higgs bundle. This is an equality of fine moduli functors on all bases. Analytification commutes with fiber products; hence $R_{\mathrm{SL}}^{\mathrm{an}}$ represents the analytic families with frames preserving the determinant trivialization.

Every point of $R_{\mathrm{SL}}$ is semistable for Simpson's linearization. Since $\SL_r$ is semisimple,~\cite[Proposition~5.5]{SimpsonModuliPartI} identifies
\begin{equation*}
R_{\mathrm{SL}}^{\mathrm{an}}\longrightarrow
(R_{\mathrm{SL}}/\!/\SL_r)^{\mathrm{an}}
\end{equation*}
as a universal categorical quotient in complex analytic spaces, preserved by analytic base change. A family in $\Fss$ has local frames preserving the determinant trivialization along $\xi$. Its maps to the framed parameter spaces, followed by the quotient maps, are independent of the frames and glue to a holomorphic classifying map to $N_m$. Theorem~\ref{thm:coarse} therefore gives $I_m\colon\cM\to N_m$ over $\cT_g$.

For the inverse, let $U\subset\cT_g$ be a sufficiently small open set lifting to $S^{\mathrm{an}}$, and put
\begin{equation*}
P_U=U\times_{S^{\mathrm{an}}}R_{\mathrm{SL}}^{\mathrm{an}}.
\end{equation*}
The map $U\to S^{\mathrm{an}}$ is locally biholomorphic, so $P_U$ is normal. Its universal family carries the marking from $U$ and is an object of $\Fss(P_U)$. Theorem~\ref{thm:coarse} gives an invariant holomorphic map $P_U\to\cM$ over $U$, which descends through the analytic quotient $P_U\to N_m|_U$ to a map $J_U\colon N_m|_U\to\cM|_U$. Both compositions with $I_m$ preserve the polystable graded object and are the identity on points. Since their source spaces are reduced, they are identities of holomorphic maps. The maps $J_U$ agree on overlaps and give the inverse of $I_m$.

The pointwise description proves uniqueness. The classifying maps of any family correspond under $I_m$ by its construction from Theorem~\ref{thm:coarse}. Changes of level, marking, and Higgs scaling act on these families, so the two compositions expressing each compatibility agree on points and hence as holomorphic maps.
\end{proof}

\subsection{The relative Hitchin map}\label{sec:relative-hitchin}

The comparison in Theorem~\ref{thm:relative-algebraic-comparison} allows us to use the relative algebraic Hitchin morphism. Let $\pi\colon\mathcal X\to\cT_g$ be the universal marked curve, and write $K_\pi$ and $K_{\pi_m}$ for the relative canonical bundles of $\pi$ and $\pi_m$. For $j\geq2$, Serre duality gives $H^1(X,K_X^j)=0$, so $\pi_*K_\pi^j$ is locally free and commutes with base change. The relative Hitchin base is
\begin{equation}\label{eq:relative-hitchin-base}
\mathcal A_r=\Tot\left(\bigoplus_{j=2}^r\pi_*K_\pi^j\right).
\end{equation}
Riemann--Roch gives $h^0(X,K_X^j)=(2j-1)(g-1)$, so its rank over $\cT_g$ is $(r^2-1)(g-1)$. We denote the characteristic coefficients by
\begin{equation}\label{eq:hitchin-coefficients}
\det(t\Id-\varphi)
=t^r+\sum_{j=2}^r a_j(\varphi)t^{r-j},
\qquad a_j(\varphi)\in H^0(X,K_X^j).
\end{equation}
The coefficient of $t^{r-1}$ is zero because $\tr\varphi=0$; for $r=1$, the sums are empty and $\mathcal A_1=\cT_g$.

\begin{prop}\label{prop:relative-hitchin-map}
The relative Hitchin map
\begin{equation}\label{eq:relative-hitchin-map}
\mathsf h\colon\cM\longrightarrow\mathcal A_r,\qquad
[X,E,\varphi]\longmapsto(X,a_2(\varphi),\ldots,a_r(\varphi))
\end{equation}
is a proper holomorphic morphism over $\cT_g$. It is equivariant for $\Gamma_g$ and for Higgs scaling, with $\lambda\in\C^*$ acting on the $j$th summand of~\eqref{eq:relative-hitchin-base} by multiplication by $\lambda^j$.
\end{prop}

\begin{proof}
For $r=1$, the map is the identity. Suppose $r\geq2$. Over $S_m$, put
\begin{equation*}
\begin{aligned}
D_m^{\mathrm{GL}}&=M_{\mathrm{Dol}}(\mathcal X_m/S_m,\GL_r),\\
A_m&=\Tot\left(\bigoplus_{j=2}^r(\pi_m)_*K_{\pi_m}^j\right),\qquad
A_m^{\mathrm{GL}}=\Tot\left(\bigoplus_{j=1}^r(\pi_m)_*K_{\pi_m}^j\right).
\end{aligned}
\end{equation*}
Characteristic coefficients define a natural transformation on families over all $S_m$-schemes. The principal $\SL_r$ description from Section~\ref{sec:actions} and the universal corepresenting property in~\cite[Theorem~9.6 and Proposition~9.7]{SimpsonModuliPartII} therefore give $\mathsf h_m\colon D_m\to A_m$.

Extension of structure group induces a finite morphism $j_m\colon D_m\to D_m^{\mathrm{GL}}$ by~\cite[Corollary~9.19]{SimpsonModuliPartII}. On a smooth curve, purity implies local freeness, and a rank-$r$ bundle of degree zero has Hilbert polynomial $rP_0$, where $P_0$ is that of the structure sheaf. Hence $D_m^{\mathrm{GL}}=M_{\mathrm{Higgs}}(\mathcal X_m/S_m,rP_0)$, and~\cite[Theorem~6.11]{SimpsonModuliPartII} gives a proper morphism $\mathsf h_m^{\mathrm{GL}}\colon D_m^{\mathrm{GL}}\to A_m^{\mathrm{GL}}$. Let $i_m\colon A_m\hookrightarrow A_m^{\mathrm{GL}}$ set the first coefficient equal to zero. Compatibility on families gives an equality of morphisms
\begin{equation}\label{eq:hitchin-factorization}
i_m\circ\mathsf h_m=\mathsf h_m^{\mathrm{GL}}\circ j_m.
\end{equation}
The right-hand side is proper. Since $i_m$ is a monomorphism and the factorization~\eqref{eq:hitchin-factorization} exists,
\begin{equation*}
D_m\times_{A_m^{\mathrm{GL}}}A_m\simeq D_m.
\end{equation*}
Thus $\mathsf h_m$ is its base change along $i_m$ and is proper.

Cohomology and base change give $\mathcal A_r\simeq\cT_g\times_{S_m^{\mathrm{an}}}A_m^{\mathrm{an}}$. Under~\eqref{eq:relative-comparison}, the map $\mathsf h$ is the analytic base change of $\mathsf h_m$. Analytification and analytic base change preserve properness, so $\mathsf h$ is holomorphic and proper. Finally, $a_j(\lambda\varphi)=\lambda^j a_j(\varphi)$, and characteristic coefficients commute with pullback of marked curves, proving the two equivariance assertions.
\end{proof}

\end{document}